\documentclass[12pt]{article}

\usepackage{wrapfig}
\usepackage{amsmath}
\usepackage{amsfonts}
\usepackage{soul}
\usepackage{color}
\usepackage{amssymb}
\usepackage{graphicx}
\usepackage{enumerate}
\usepackage{amsthm,amscd}
\usepackage[all]{xy}
\usepackage{ulem}
\usepackage{comment}
\usepackage{anyfontsize}

 \allowdisplaybreaks

\usepackage{hyperref}

\newtheorem{theorem}{Theorem}[section]

\newtheorem{question}[theorem]{Question}
\newtheorem{lemma}[theorem]{Lemma}
\newtheorem{proposition}[theorem]{Proposition}

\newtheorem{clm}[theorem]{Claim}
\newtheorem{fact}[theorem]{Fact}
\newtheorem{remark}[theorem]{Remark}
\newtheorem{construction}[theorem]{Construction}

\begin{document}

\title{Stability of intersecting families\thanks{This work is supported by  NSFC (Grant No. 11931002).  E-mail addresses: 1060393815@qq.com (Yang Huang),
ypeng1@hnu.edu.cn (Yuejian Peng, corresponding author).}}

\author{Yang Huang, Yuejian Peng$^{\dag}$ \\[2ex]
{\small School of Mathematics, Hunan University} \\
{\small Changsha, Hunan, 410082, P.R. China }  }

\maketitle

\vspace{-0.5cm}

\begin{abstract}
The celebrated Erd\H{o}s--Ko--Rado theorem \cite{EKR1961} states that the maximum intersecting $k$-uniform family on $[n]$ is a full star if $n\ge 2k+1$. Furthermore,  Hilton-Milner \cite{HM1967} showed that  if  an intersecting $k$-uniform family on $[n]$  is not a subfamily of a full star,  then its maximum size achieves only on a family isomorphic to
$HM(n,k):=
\Bigl\{G\in {[n] \choose k}:
1\in G, G\cap [2,k+1] \neq \emptyset \Bigr\}
\cup \Bigl\{ [2,k+1]  \Bigr\}
$
if $n>2k$ and $k\ge 4$, and there is one more possibility in the case of $k=3$.  Han and Kohayakawa \cite{HK2017} determined the maximum intersecting $k$-uniform family on $[n]$ which is neither a subfamily of a full star nor a subfamily of the extremal family in Hilton-Milner theorm, and they asked what is the next maximum intersecting $k$-uniform family on $[n]$.  Kostochka and Mubayi \cite{KM2016} gave the answer  for  large enough $n$. In this paper, we are going to get rid of  the requirement that $n$ is large enough in the result by  Kostochka and Mubayi \cite{KM2016} and answer the question of Han and Kohayakawa \cite{HK2017}.
 \end{abstract}



\section{Introduction}

    For a positive interge $n$, let $[n]=\{1, 2, \dots, n\}$ and $2^{[n]}$ be the family of all subsets of $[n]$.  An $i$-element subset $A\subseteq [n]$ is called an $i$-set.
    % For two families $\mathcal{H} \in 2^{[n]}$ and $\mathcal{G} \in 2^{[n]}$, we write $\mathcal{H}\subseteq_s \mathcal{G}$ if there is an injective map $f: [n] \to [n]$ such that $f(E)\in  \mathcal{G}$ for every $E \in  \mathcal{H}$. We say that $\mathcal{H}$ and $\mathcal{G}$ are {\it isomorphic} if $\mathcal{H}\subseteq_s \mathcal{G}$ and $\mathcal{G}\subseteq_s \mathcal{H}$. We sometimes write $\mathcal{H}=\mathcal{G}$ for isomorphic $\mathcal{H}$ and $\mathcal{G}$ if there is no risk of confusion.
    For $0\le k \le n$, let ${[n] \choose k}$ denote the collection of all $k$-sets of $[n]$. A family $\mathcal{F} \subseteq {[n] \choose k}$ is called {\it k-uniform}. For a family $\mathcal{F}\subseteq 2^{[n]}$, we say $\mathcal{F}$ is {\it intersecting}  if for any two distinct sets $F$ and $F'$ in $\mathcal{F}$ we have $ | F\cap F' |\ge 1$. In this paper, we always consider a $k$-uniform intersecting family on $[n]$. The following celebrated theorem of
Erd\H{o}s--Ko--Rado determines the maximum intersecting family.

For $x\in [n]$ denote $\mathcal{F}_x:=\{F\in {[n] \choose k}: x\in F\}$ by the {\it full star} centered at $x$. We say $\mathcal{F}$ is {\it EKR} if $\mathcal{F}$ is contained in a full star.
\begin{theorem}[Erd\H{o}s--Ko--Rado \cite{EKR1961}]\label{thm1.1}
Let $n\ge 2k$ be integer and $\mathcal{F}$ be a $k$-uniform intersecting family of subsets of $[n]$. Then
$$  |\mathcal{F}| \le {n-1 \choose k-1}.$$
  Moreover, when $n>2k$,
  equality holds if and only if
  $\mathcal{F}$ is a full star.
\end{theorem}

The theorem of Hilton-Milner determines the maximum size of non-EKR families.

\begin{theorem}[Hilton--Milner \cite{HM1967}] \label{thm1.2}
Let $k\ge 2$ and $n\ge 2k$ be integers and $\mathcal{F}\subseteq {[n] \choose k}$ be
an intersecting  family.
 If  $\mathcal{F}$ is not EKR,  then

$$ |\mathcal{F}| \le {n-1 \choose k-1} -{n-k-1 \choose k-1} +1.$$
 Moreover, for $n>2k$  and $k\ge 4$, equality holds if and only if
$\mathcal{F}$ is isomorphic to
\[ \begin{matrix}
HM(n,k):=
\Bigl\{G\in {[n] \choose k}:
1\in G, G\cap [2,k+1] \neq \emptyset \Bigr\}
\cup \Bigl\{ [2,k+1]  \Bigr\}.
\end{matrix} \]
 For the case  $k=3$,
there is one more possibility, namely
\[ \begin{matrix}
\mathcal{T}(n,3):=\left\{F\in {[n] \choose 3}:
|F\cap [3]|\ge 2 \right\}.
\end{matrix} \]
\end{theorem}

We say a family $\mathcal{F}$ is {\it HM} if it is isomorphic to a subfamily of HM($n, k$). We say that 1 is the {\it center} of HM($n, k$).

Let $E\subseteq [n]$ be an $i$-set and $x\in [n]$. We define
\[
\mathcal{G}_i:= \left\{ G\in {[n]\choose k}:E\subseteq G\right\} \cup \left\{ G\in {[n]\choose k}: x\in G \,\, \text{and} \,\, G\cap E\neq \emptyset \right\}.
\]
We call $x$ the {\it center}, and $E$ the {\it core} of $\mathcal{G}_i$ for $i\ge 3$.
%If $\mathcal{F} \subseteq_s \mathcal{G}_i ( i\ge3 )$, then we say that $\mathcal{F}$ is $\mathcal{G}_i$ at $x$.
With a slight tweaking, we call $\{x\}\cup E$ the {\it core} of $\mathcal{G}_2$.
%if $\mathcal{F} \subseteq_s \mathcal{G}_2$, then we say that $\mathcal{F}$ is $\mathcal{G}_2$ at a 3-set $ A$ where $A=\{x\} \cup E$   .
 Note that $\mathcal{G}_k=HM(n, k)$.

For a $(k-1)$-set $E$, a point $x\in [n]\setminus E$, and an $i$-set $J \subset [n] \setminus (E\cup\{x\})$, we denote
\begin{align*}
\mathcal {J}_i:=\left\{G\in {[n]\choose k}: E \subseteq G \,\, \text{and} \,\, G \cap J \neq \emptyset\right\} \cup
\left\{G\in {[n]\choose k}:J\cup\{x\}\subseteq G\right\} \\
\cup \left\{G\in {[n]\choose k}:x\in G, G\cap E \neq \emptyset\right\}.
\end{align*}
We call $x$ the {\it center}, $E$ the {\it kernel}, and $J$ the {\it set of pages}.
%We say that $\mathcal{F}$ is $\mathcal{J}_2\,  (\,\text{or}\, \mathcal{J}_3\, )$ at $x$ if $\mathcal{F} \subseteq_s \mathcal{J}_2 \, ( \,\text{or}\, \mathcal{J}_3\, )$.

For two $k$-sets $E_1$ and $E_2 \subseteq [n] $ with $|E_1\cap E_2|=k-2$, and  $x\in [n]\setminus (E_1\cup E_2)$, we define
$$\mathcal{K}_2:=\{ G\in {[n]\choose k}: x\in G, G\cap E_1\ne \emptyset \,\, \text{and} \,\, G\cap E_2\ne \emptyset \} \cup \{E_1, E_2\},$$
and call $x$ the {\it center} of $\mathcal{K}_2$.

In \cite{HK2017}, Han and Kohayakawa obtained the size of a maximum non-EKR, non-HM intersecting family.

\begin{theorem}[Han--Kohayakawa \cite{HK2017}]\label{thm1.4}

Suppose $k\ge 3$ and $n\ge 2k+1$ and let $\mathcal{H}$ be an intersecting $k$-uniform family on
 $[n]$. Furthermore, assume that $\mathcal{H}$ is neither EKR nor HM, if $k=3$,
  $\mathcal{H} \not\subseteq \mathcal{G}_2
$. Then
  $$|\mathcal{H}|\le {n-1 \choose k-1} -{n-k-1 \choose k-1} -{n-k-2 \choose k-2}+2.$$
  For $ k=4$, equality holds if and only if $\mathcal{H}= \mathcal{J}_{2},\, \mathcal{G}_2$ or
  $ \mathcal{G}_3$. For every other $k$, equality holds if and only if $\mathcal{H}= \mathcal{J}_2 $.
\end{theorem}

 Han and Kohayakawa \cite{HK2017} proposed the following question.

 \begin{question}\label{quest1}
Let $n\ge 2k+1$. What is the maximum size of an intersecting family $\mathcal{H}$ that is neither EKR nor HM, and $\mathcal{H} \not\subseteq \mathcal{J}_2$ $(\!\!$ in addition $\mathcal{H} \not\subseteq \mathcal{G}_2$ and $\mathcal{H} \not\subseteq \mathcal{G}_3$ if $k=4$ $\!\!)$?
 \end{question}

  Regarding this question, Kostochka and Mubayi \cite{KM2016} showed that the answer is $|\mathcal{J}_3|$ for sufficiently large $n$. In fact they proved that the maximum size of an intersecting family that is neither EKR, nor HM, nor contained in $\mathcal{J}_i$ for each $i$, $2\le i\le k-1$ (nor in $\mathcal{G}_2,\mathcal{G}_3$ for $k=4$) is $|\mathcal{K}_2|$ for all large enough $n$. In paper \cite{KM2016}, they also established the structure of almost all intersecting 3-uniform families.
 % that when $k=3$, the maximun size is $n+4$ with 3 extremal families.
Sometimes, it is relatively easier to get extremal families under the assumption that $n$ is large enough. For example, %$t$-intersecting problem and matching conjecture. The former  for sufficiently large was solved by Erd\H{o}s, Ko, and Rado in 1967, however, the complete answer for all $n>2k-t$ was obtained in 1997 by Ahlswede and Khachatrian.
  Erd\H{o}s matching conjecture \cite{MC1965} states  that for a $k$-uniform family $\mathcal{F}$ on finite set $[n]$,  $|\mathcal{F}|\le \max\{{k(s+1)-1\choose k}, {n\choose k}-{n-s\choose k}\}$ if there is no $s+1$ pairwise disjoint members of $\mathcal{F}$ and $n\ge(s+1)k$, and it was proved to be true for large enough  $n$ in \cite{MC1965}. There has been a lot of recent studies for small n (see \cite{Fra2013, FraX, HLS2012, LM2014}).
%We simply describe the idea of the original proof. Suppose on the contrary that $|\mathcal{F}|>{n\choose k}-{n-s\choose k}$. Let $v$ be a point of maximum degree. Then $|\mathcal{F}-v|\ge |\mathcal{F}-{n-1\choose k-1}|>{n-1\choose k}-{n-1-(s-1)\choose k}$. By induction on $s$, we can find $s$ disjoint members $F_1, \dots, F_s$ such that they are all missing $v$. Notice that the number of members containing $v$ and meeting at least one point from $\cup_{i=1}^{s}F_i$ is at most $ks{n-2\choose k-2}$. If $deg(v)>ks{n-2\choose k-2}$, then $\nu({\mathcal{F}})\ge s+1$, a contradiction. Otherwise, the number of members meeting $\cup_{i=1}^{s}F_i$ is at most $|\cup_{i=1}^{s}F_i|\cdot ks {n-2\choose k-2}=k^2s^2{n-2\choose k-2}<{n\choose k}-{n-s\choose k}$ for large $n$. This also implies that $\nu({\mathcal{F}})\ge s+1$, a contradiction. This makes the conjecture is tree for $n$ sufficiently large.
However, the conjecture is not completely verified for small $n$.  Up to now, the best condition on $n$  was given by Frankl in \cite{Fra2013, FT2018} that $n\ge k(2s+1)-s$, for $(s+1)k\le n\le k(2s+1)-s-1$.

 As mentioned by Han and Kohayakawa in  \cite{HK2017}, for $k\ge 4$, the bound in Theorem \ref{thm1.4} can be deduced from Theorem 3 in \cite{HM1967} which was established by Hilton and Milner in 1967. However,  family $\mathcal{H}$ in Question \ref{quest1}  does not satisfy the hypothesis of Theorem 3 in \cite{HM1967} for  $k\ge 4$. This makes Question \ref{quest1} more interesting.
In this paper, we answer Question \ref{quest1}. We are going to get rid of  the requirement that $n$ is large enough in the result by  Kostochka and Mubayi \cite{KM2016}.  As in the proofs of Theorem \ref{thm1.1}, Theorem \ref{thm1.2} and Theorem \ref{thm1.4}, we will apply the shifting method. The main difficulty in our proof is to guarantee that we can get a {\it stable} family which is not EKR, not HM, $\not\subseteq \mathcal{J}_2$ $(\!\!$ in addition $\not\subseteq \mathcal{G}_2, \not\subseteq \mathcal{G}_3$ if $k=4$ $\!\!)$ after performing a series of shifts to a family which is not EKR, not HM, $\not\subseteq \mathcal{J}_2$ $(\!\!$ in addition $\not\subseteq \mathcal{G}_2, \not\subseteq \mathcal{G}_3$ if $k=4$ $\!\!)$.
  Our main result is as follows.

\begin{theorem}\label{main}

Let $k\ge 4$ and $\mathcal{H} \subseteq {[n]\choose k}$ be an intersecting  family  which is neither EKR nor HM. Furthermore, $\mathcal{H} \not\subseteq \mathcal{J}_2$ $(\!\!$ in addition $\mathcal{H} \not\subseteq \mathcal{G}_2$ and $\mathcal{H} \not\subseteq \mathcal{G}_3$ if $k=4$ $\!\!)$. \\
(i) If $2k+1\le n\le 3k-3$, then
  $$|\mathcal{H}| \le {n-1 \choose k-1} -2{n-k-1 \choose k-1} +{n-k-3 \choose k-1}+2,$$
  Moreover,  the equality holds only for $\mathcal{H}= \mathcal{K}_2 $ if $k\ge 5$, and $\mathcal{H}= \mathcal{K}_2 $ or $\mathcal{J}_3$ if $k=4$.\\
(ii) If  $n\ge 3k-2$, then
  $$|\mathcal{H}|\le {n-1 \choose k-1} -{n-k-1 \choose k-1} -{n-k-2 \choose k-2} -{n-k-3 \choose k-3}+3.$$
 Moreover, for $k=5$, the equality holds only for $\mathcal{H}= \mathcal{J}_{3}$  or $\mathcal{G}_4$. For every other $k$, equality holds only for $\mathcal{H}= \mathcal{J}_3 $.
\end{theorem}

In Section \ref{sec2}, we will give the proof of Theorem \ref{main}. The proofs of some crucial lemmas for the proof of Theorem \ref{main} are given in Section \ref{sec3}.

\section{Proof of Theorem \ref{main}}\label{sec2}

 In this section, we always assume that $\mathcal{H}$ is a maximum intersecting family which satisfies the conditions of Theorem  \ref{main}, that is, $\mathcal{H}$ is not EKR, not HM, $\mathcal{H} \not\subseteq \mathcal{J}_2$ $(\!\!$ in addition $\mathcal{H} \not\subseteq \mathcal{G}_2, \mathcal{H} \not\subseteq \mathcal{G}_3$ if $k=4$ $\!\!)$.  By direct calculation, we have the following fact.
\begin{fact}\label{fact0}
$ (i) $ Suppose that  there is $x\in [n]$ such that there are only 2 sets, say, $E_1\,\, \text{and} \,\, E_2 \in \mathcal{H}$ missing $x$. If $|E_1 \cap E_2|=k-i \,\, \text{and} \,\, i\ge 2$,  then
\begin{align}\label{eq1}
|\mathcal{H}| &\le {n-1 \choose k-1} -2{n-k-1 \choose k-1} +{n-k-i-1 \choose k-1}+2 \nonumber\\
&\le {n-1 \choose k-1} -2{n-k-1 \choose k-1} +{n-k-3 \choose k-1}+2.
\end{align}
 The equality in (\ref{eq1}) holds if and only if $|E_1 \cap E_2|=k-2$, that is $\mathcal{H}=\mathcal{K}_2$.\\
$ (ii) $ By the definiton of $\mathcal{J}_i$, we have
\begin{equation}\label{eq2}
|\mathcal{J}_3|={n-1 \choose k-1} -{n-k-1 \choose k-1} -{n-k-2 \choose k-2} -{n-k-3 \choose k-3}+3.
\end{equation}
$ (iii) $ Comparing the right hand sides of $($\ref{eq1}$)$ and $($\ref{eq2}$)$, we can see that if  $2k+1\le n \le 3k-3$, then $|\mathcal{K}_2|\ge |\mathcal{J}_3|$, the equality holds only for $k=4$; and if  $n \ge 3k-2$, then $|\mathcal{K}_2|< |\mathcal{J}_3|$.
\end{fact}

By Fact \ref{fact0}, we may assume that for any $x$, at least 3 sets in $\mathcal{H}$ do not contain $x$. To show Theorem \ref{main}, it is sufficient to show the following result.

\begin{theorem}\label{thm1.5}
Let $k\ge 4, n\ge 2k+1$ and $\mathcal{H} \subseteq {[n]\choose k}$ be an intersecting  family  which is not EKR, not HM and $\mathcal{H} \not\subseteq \mathcal{J}_2$ $(\!\!$ in addition $\mathcal{H} \not\subseteq \mathcal{G}_2, \mathcal{H} \not\subseteq \mathcal{G}_3$ if $k=4$ $\!\!)$. Moreover, for any $x\in [n]$, there are at least 3 sets in $\mathcal{H}$ not containing $x$. Then
$$|\mathcal{H}|\le {n-1 \choose k-1} -{n-k-1 \choose k-1} -{n-k-2 \choose k-2} -{n-k-3 \choose k-3}+3.$$ Moreover if $k\ne 5$, the equality holds only for $\mathcal{H}=\mathcal{J}_3$; if $k=5$, the equality holds for $\mathcal{H}=\mathcal{J}_3$ or $\mathcal{G}_4$.
\end{theorem}

From now on, we always assume that $\mathcal{H}$ is a maximum intersecting family which satisfies the conditions of Theorem  \ref{thm1.5}, that is   $\mathcal{H}$ is not EKR, not HM, $\mathcal{H} \not\subseteq \mathcal{J}_2$ $(\!\!$ in addition $\mathcal{H} \not\subseteq \mathcal{G}_2, \mathcal{H} \not\subseteq \mathcal{G}_3$ if $k=4$ $\!\!)$ and for any $x\in [n]$, there are at least 3 sets in $\mathcal{H}$ not containing $x$.

We first give some definition related to the shifting method. For $x$ and $y \in [n], x<y$, and $F\in \mathcal{F}$,  we call the following operation a {\it shift}\,:
 $$
 S_{xy}(F)=
  \begin{cases}
  (F\setminus \{y\})\cup \{x\}, & \text{if}  \  x \not \in F, y\in F\, \text{and} \,(F\setminus \{y\})\cup \{x\} \not \in \mathcal{F};\\
  F,& \text{otherwise}.
  \end{cases}
  $$
We say that $F$ is {\it stable} under the shift $S_{xy}$ if $S_{xy}(F)=F$. If $z\in F$ and $z\in S_{xy}(F)$ still, we say that $F$ is {\it stable at $z$} after the shift $S_{xy}$.
For a family $\mathcal{F}$, we define
$$S_{xy}(\mathcal{F})=\{S_{xy}(F): F\in \mathcal{F}\}.$$
Clearly, $\vert S_{xy}(\mathcal{F})\vert = \vert\mathcal{F}\vert$. We say that $\mathcal{F}$ is {\it stable}  if $S_{xy}(\mathcal{F})= \mathcal{F}$ for all $x, y \in [n]$ with $x<y$.

An important property shown in \cite{Fra1987shift} is that if $\mathcal{F}$ is intersecting, then $S_{xy}(\mathcal{F})$ is still intersecting.  Let us rewrite is as a remark.

\begin{remark} \cite{Fra1987shift}\label{shiftpro}
If $\mathcal{F}$ is a maximum intersecting family, then $S_{xy}(\mathcal{F})$ is still a maximum intersecting family.
\end{remark}

 This property guarantees that performing shifts repeatedly to a maximum intersecting family will yield a stable maximum intersecting family. The main difficulty we need to overcome  is to guarantee that we can get a  stable maximum intersecting family with further properties:  not EKR, not HM, $\not\subseteq \mathcal{J}_2$ $(\!\!$ in addition $\not\subseteq \mathcal{G}_2, \not\subseteq \mathcal{G}_3$ if $k=4$ $\!\!)$. The following facts and lemmas are for this purpose.

\begin{fact} \label{fact 2}
The following properties hold. \\
(i) If $S_{xy}(\mathcal{H})$ is EKR $(\!\!$ or HM $\!)$, then $x$ must be the center.\\
(ii) If $S_{xy}(\mathcal{H})\subseteq\mathcal{G}_2$, then the core is
%$x$ is the center or $x$ in the core, and $y$ is neither center nor in the core.\\
 $\{ x, x_1, x_2 \}$ for some $x_1, x_2 \in [n] \setminus \{x, y\}$.\\
(iii) If $S_{xy}(\mathcal{H})\subseteq\mathcal{J}_2$, then $x$ is the center.\\
% or $x$ in the kernal. Moreover, if $x$ in the kernal, then $\mathcal{H} \subseteq \mathcal{K}_2$.\\
(iv) If $S_{xy}(\mathcal{H})\subseteq\mathcal{G}_3$, then $x$ is the center or $x$ is in the core.
%Moreover, if $x$ is not the center, then $y$ is neither center nor in the core.
\end{fact}
%%%%%但是这种情况下，x不一定是中心点，也可能在E中，且y也有可能在B中%%%
\begin{proof}
For (i) and (ii), Han and Kohaykawa proved them in \cite{HK2017}. We prove (iii) and (iv) only.

For (iii), suppose that $S_{xy}(\mathcal{H})\subseteq\mathcal{J}_2$ at center $z \in [n] \setminus \{x\}$. Since $\mathcal{H}\not\subseteq\mathcal{J}_2$ at $z$,  there are at least three sets
$E_1, E_2$ and $E_3$ in $\mathcal{H}$ missing $z$,  after doing the shift $S_{xy}$, these 3 sets still miss $z$, so $S_{xy} (\mathcal{H})$ is not contained in $\mathcal{J}_2$ center at $z$.

For (iv), let $S_{xy}(\mathcal{H})\subseteq\mathcal{G}_3$ at center $x_0$ and core $E=\{x_1, x_2, x_3\}$, and let $B=\{x_0, x_1, x_2, x_3\}$. Since $\mathcal{H}\not \subseteq \mathcal{G}_3$, there is a set $G \in \mathcal{H} $ that satisfies one of the following two cases:
(a) $\{y, x_0\} \subseteq G, G \cap E=\emptyset$; (b) $y \in G, x_0 \not \in G, |G \cap E| \in \{1, 2\}$.
If (a) holds, then $x\ne x_0$ and $x$ must be in the core, $y\not \in B$. If (b) holds,
then either $x=x_0$ is the center or $x$ is in the core and $y\not \in B$.
\end{proof}

\begin{remark}\label{remark2.4}
By Fact \ref{fact 2},  if applying $S_{x'y'} (x' < y')$ repeatedly to $\mathcal {H}$, we may reach a family which belong to one of the following  cases.\\
Case 1: a family $\mathcal{H}_1$ such that $S_{xy}(\mathcal{H}_1)$ is EKR with center $x$;\\
Case 2: a family $\mathcal{H}_2$ such that $S_{xy}(\mathcal{H}_2)$ is HM with center $x$;\\
Case 3: a family $\mathcal{H}_3$ such that $S_{xy}(\mathcal{H}_3)\subseteq\mathcal{J}_2$ with center $x$;\\
Case 4: a family $\mathcal{H}_4$ such that $S_{xy}(\mathcal{H}_4)\subseteq\mathcal{G}_2$ with core $\{x, x_1, x_2\}$ for some $\{x_1, x_2\} \in X\setminus \{x, y\}$ $($$k=4$ only$)$;\\
Case 5: a family $\mathcal{H}_5$ such that $S_{xy}(\mathcal{H}_5)\subseteq\mathcal{G}_3$ with center $x$ or $x$ being in the core $($$k=4$ only$)$;\\
Case 6: a stable family $\mathcal{H}_6$ satisfies the conditions of Theorem \ref{thm1.5}, that is we will not meet Cases 1-5 after doing all shifts.
\end{remark}

By Remark \ref{shiftpro}, we know that for any shift $S_{xy}$ on $[n]$ we have $|S_{xy}(\mathcal{H})|=|\mathcal{H}|$ and $S_{xy}(\mathcal{H})$ is also intersecting.  We hope to get a stable family satisfying the conditions of Theorem \ref{thm1.5} after some shifts, that is neither EKR, nor HM, nor contained in $\mathcal{J}_2$ (nor in $\mathcal{G}_2$, $\mathcal{G}_3$ if $k=4$). By Fact \ref{fact0}, we can  assume that a family $\mathcal{G}$ obtained by performing shifts to $\mathcal{H}$ has the property that for any $x$, at least 3 sets in $\mathcal{G}$ do not contain $x$. What we are going to do is: If any case of Cases 1-5 happens, we will not perform $S_{xy}$. Instead we will adjust the shifts as shown in Lemma \ref{lem2.1} to guarantee that the terminating family is a stable family satisfying the conditions of Theorem \ref{thm1.5}. We will prove the following two crucial lemmas  in Section \ref{sec3}.

\begin{lemma}\label{lem2.1}
Let $ i \in [5] $. If we reach $\mathcal{H}_i$ in {\it Case i } in Remark \ref{remark2.4}, then there is a set $X_i \subseteq [n]$ with $|X_i| \le 5$ $($when $k\ge 5$, $|X_i| \le 3$ for  $i\in [3]$$)$, such that after a series of shifts $S_{x'y'} \, (x'<y'\,\,\text{and}\,\, x', y' \in [n]\setminus X_i)$ to $\mathcal{H}_i$, we can reach a stable family satisfying the conditions of Theorem \ref{thm1.5}. Moreover, for any set $G$ in the final family $\mathcal{G}$, we have $G \cap X_i \ne \emptyset$.
\end{lemma}

From now on, let $X_i$ be the corresponding sets in Lemma \ref{lem2.1} for $1\le i \le 5$ and $X_6=\emptyset$ .  For $k\ge 5$ and $i\in \{1, 2, 3, 6\}$, let $Y_i$ be the set of the first $2k-|X_i|$ elements of $[n]\setminus X_i$, and for $k=4$ and $i\in \{1, 2, 3, 4, 5, 6\}$, let $Y_i$ be the first $9-|X_i|$ elements of $[n]\setminus X_i$. Let $Y=Y_i\cup X_i$, then $|Y_i|\ge 2k-4$ and $|Y|=2k$ if $k\ge 5$. If $k=4$ then  $|Y|=9$.
 Let
\begin{align*}
\mathcal{A}_i:&=\{G\cap Y: G\in \mathcal{G}, |G\cap Y|=i\},\\
\widetilde{\mathcal{A}_i}:&=\{G: G\in \mathcal{G}, |G\cap Y|=i\}.
\end{align*}

\begin{lemma} \label{lem2.2}
Let $\mathcal{G}$ be the final stable family guaranteed by Lemma \ref{lem2.1} satisfying the conditions of Theorem \ref{thm1.5}, and let $X_i$ be inherit from Lemma \ref{lem2.1}. In other words, $\mathcal{G}$ is stable;   $\mathcal{G}$ is neither EKR, nor HM, nor contained in $\mathcal{J}_2$ $($nor in $\mathcal{G}_2$, $\mathcal{G}_3$ if $k=4$$)$; for any $x\in [n]$, there are at least 3 sets in $\mathcal{G}$ not containing $x$; and $G\cap X_i\ne \emptyset$ for any $G\in \mathcal{G}$.
%\[ G\cap X_i\ne \emptyset \, \text { for any} \,\, G\in \mathcal{G},\]
%and
%\[ S_{x'y'}(\mathcal{G})=\mathcal{G} \, \text { for any} \,\, x'<y'\,\,\text{and}\,\, x', y'\in [n]\setminus X_i. \]
Then \\
$(i)$ $\mathcal{A}_1=\emptyset$.\\
$(ii)$ For all $G$ and $G'\in \mathcal{G}$, we have $G\cap G' \cap Y\ne \emptyset$, or equivalently, $\cup_{i=2}^{k}\mathcal{A}_i\cup \mathcal{G}$ is intersecting.
\end{lemma}

%第2.1 节.
\subsection{Quantitative Part of Theorem \ref{thm1.5}}
\begin{lemma}\label{lem2.3}
 For $k=4$, we have $|\mathcal{A}_1|=0$, $| \mathcal{A}_2 | \le 3, | \mathcal{A}_3 | \le 18$ and $|\mathcal{A}_4 | \le 50$. For $k\ge 5$, we have
$$| \mathcal{A}_i | \le {2k-1 \choose i-1}-{k-1 \choose i-1}-{k-2 \choose i-2}-{k-3 \choose i-3},  \,\, 1\le i \le k-1, $$
$$| \mathcal{A}_k | \le {\frac{1}{2}}{2k \choose k}={2k-1 \choose k-1}-{k-1 \choose k-1}-{k-2 \choose k-2}-{k-3 \choose k-3}+ 3.$$
\end{lemma}

\begin{proof}
By Lemma \ref{lem2.2} (i), we have $|\mathcal{A}_1|=0$.

First consider $k=4$. If $|\mathcal{A}_2| \ge 4$, since $\mathcal{A}_2$ is intersecting, it must be a star. Let its center be $x$. Since $\mathcal{A}_2\cup \mathcal{A}_3 \cup \mathcal{A}_4$ is intersecting,  $\mathcal{A}_3$ must be a star with center $x$ and there is at most one set in $\mathcal{A}_4$  missing $x$, this implies that $\mathcal{G}$ is EKR or HM, which contradicts the fact that $\mathcal{G}$ is neither EKR nor HM.

Suppose that $ | \mathcal{A}_3 | \ge 19$. By Theorem \ref{thm1.4}, $\mathcal{A}_3$ must be EKR, HM or $\mathcal{G}_2$.

If $\mathcal{A}_3$ is EKR with center $x$, then since $\mathcal{G}$ is not EKR and $\mathcal{A}_1=\emptyset$,
there must exist $G\in \mathcal{G}$, such that either $x\not\in G$ and $G\cap Y\in \mathcal{A}_2$, or $x\not\in G$ and $G\cap Y\in \mathcal{A}_4$. If the former holds, by the intersecting property of $\mathcal{A}_2\cup \mathcal{A}_3$, every set in $\mathcal{A}_3$ must contain at least one of the elements in $G\cap Y$, so $ | \mathcal{A}_3 | \le 13$, a contradiction. Otherwise, the latter holds and $\mathcal{A}_2$ is a star with center $x$, and all sets of $\mathcal{G}$ missing $x$ lie in $Y$ completely. Recall that the number of these sets is at leat 3, say $x\not\in G_1, G_2, G_3\in \mathcal{G}$. Since $\mathcal{G}$ is not $\mathcal{G}_3$, it's impossible that $G_1, G_2, G_3$  form a 3-star (each member contains a fixed 3-set). If any two sets in $G_1, G_2, G_3$   intersect at $3$ vertices, then $G_1, G_2, G_3$ must be a 2-star. Since $\mathcal{A}_3\cup \mathcal{A}_4$ is intersecting, calculating directly the number of triples of $Y$ containing $x$ and intersecting with $G_1, G_2$ and $G_3$, we have $ | \mathcal{A}_3 | \le 16$, a contradiction. Otherwise, there are two members, w.l.o.g., say, $G_1, G_2$, such that $|G_1\cap G_2|=2$. Since $\mathcal{A}_3\cup \mathcal{A}_4$ is intersecting,  calculating directly  the number of triples of $Y$ containing $x$ and intersecting with $G_1$ and $G_2$, we have $ | \mathcal{A}_3 | \le 17$, also a contradiction. %实际上即使是3-full star 也没关系，至少有3个不含x的sets，总会使得 | \mathcal{A}_3 | \le 18.

If $\mathcal{A}_3$ is HM with center $x$, let $\{z_1, z_2, z_3\}\in \mathcal{A}_3$. By Theorem \ref{thm1.2}, we have $ | \mathcal{A}_3 | \le 19$,
so we may assume $ | \mathcal{A}_3 | = 19$ and $\mathcal{A}_3$ is isomorphic to $HM(9, 3)$. Suppose that there is a set $G$ such that $x\not\in G, G\cap Y\in \mathcal{A}_2$, w.l.o.g., assume $z_1\not\in G$. Since $|Y\setminus(\{x, z_1, z_2, z_3\}\cup G)|\ge 3$,  there is $a\in Y\setminus(\{x, z_1, z_2, z_3\}\cup G)$ such that $\{x, z_1, a\}\cap G=\emptyset$. By the intersecting property of $\mathcal{A}_3\cup \mathcal{A}_4$, we have $\{x, z_1, a\}\not\in \mathcal{A}_3$, so  $ | \mathcal{A}_3 | < 19$, a contradiction. Now we may assume that $\mathcal{A}_2$ is a star with center $x$. Since $\mathcal{G}$ is neither HM nor contained in $\mathcal{G}_3$, there must be a 4-set $G$ in $\mathcal{A}_4$ such that either $x\not\in G$ and $1\le |G\cap \{z_1, z_2, z_3\}|\le 2$, w.l.o.g., assume $z_1\not\in G$ or $x\in G$ and $|G\cap \{z_1, z_2, z_3\}|=0$. But since   $\mathcal{A}_3\cup \mathcal{A}_4$ is intersecting, the latter case will not happen. Assume the former holds. Since $|Y\setminus(\{x, z_1, z_2, z_3\}\cup G)|\ge 2$,  there is $a\in Y\setminus(\{x, z_1, z_2, z_3\}\cup G)$ such that $\{x, z_1, a\}\cap G=\emptyset$. By the intersecting property of $\mathcal{A}_3\cup \mathcal{A}_4$, we have $\{x, z_1, a\}\not\in \mathcal{A}_3$, so  $ | \mathcal{A}_3 | < 19$.

At last, assume that $\mathcal{A}_3 \subseteq \mathcal{G}_2$ with core $\{x_1, x_2, x_3\}$. Since $  \mathcal{A}_3$ is intersecting, by calculating the number of triples in $Y$ containing at least 2 vertices in core $\{x_1, x_2, x_3\}$, we have $ | \mathcal{A}_3 | \le 19$, so we may assume that $ | \mathcal{A}_3 | = 19$. Since $\mathcal{G}\not\subseteq\mathcal{G}_2$, there exists a set $G\in\mathcal{G}$ such that $| G\cap \{x_1, x_2, x_3\}  |\le 1$. w.l.o.g., let $G\cap \{x_1, x_2\}=\emptyset$. Since $ |Y\setminus (\{x_1, x_2, x_3\}\cup G)|\ge 2 $, we can pick $a\in Y\setminus (\{x_1, x_2, x_3\}\cup G)$ such that $G\cap \{x_1, x_2, a\}=\emptyset$. By the intersecting property of $\mathcal{A}_3\cup \mathcal{G}$, we have $\{x_1, x_2, a\}\not\in \mathcal{A}_3$, hence $ | \mathcal{A}_3 | \le 18$, as desired.

So we have proved that $ | \mathcal{A}_3 | \le 18$ for $k=4$.

Next, we prove $ | \mathcal{A}_4 | \le 50$. On the contrary, suppose that $ | \mathcal{A}_4 | \ge 51$. By Theorem \ref{thm1.4}, $\mathcal{A}_4$ must be EKR, HM, or contained in $\mathcal{J}_2$, $\mathcal{G}_2$ or $\mathcal{G}_3$.

 Suppose that $\mathcal{A}_4$ is EKR at $x$. Since $\mathcal{G}$ is not EKR and $\mathcal{A}_1=\emptyset$, there must exist $G\in \mathcal{G}$ such that either $x\not\in G$ and $G\cap Y\in \mathcal{A}_2$ or $x\not\in G$ and $G\cap Y\in \mathcal{A}_3$. If the former holds, since $\mathcal{A}_2\cup \mathcal{A}_4$ is intersecting, by calculating the number of 4-sets in $Y$ containing $x$ and intersecting with $G\cap Y$ directly, we have $ | \mathcal{A}_4 | \le 36$. If the latter holds, since $\mathcal{A}_3\cup \mathcal{A}_4$ is intersecting, by calculating the number of 4-sets in $Y$ containing $x$ and intersecting with $G\cap Y$ directly, we have $ | \mathcal{A}_4 | \le 46$.

 Suppose that $\mathcal{A}_4$ is HM at $x$. Since $\mathcal{G}$ is not HM at $x$, there exists $G\in \mathcal{G}$ such that either $x\not\in G$ and $G\cap Y\in \mathcal{A}_2$ or $x\not\in G$ and $G\cap Y\in \mathcal{A}_3$, since $\mathcal{A}_4$ is HM at $x$ and $\mathcal{A}_2\cup \mathcal{A}_4$ (or $\mathcal{A}_3\cup \mathcal{A}_4$) is intersecting, by calculating the number of 4-subsets containing $x$ and intersecting with $G\cap Y$, and adding $1$ set not containing $x$, we have $|\mathcal{A}_4|\le 37$ (or $|\mathcal{A}_4|\le 47$).

 Suppose that $\mathcal{A}_4\subseteq\mathcal{G}_2$ with core $\{x_1, x_2, x_3\}=A$. By calculating the number of 4-subsets in $Y$ containing at least 2 of $\{x_1, x_2, x_3\}$, we have $ | \mathcal{A}_4 | \le 51$,  so we may assume $ | \mathcal{A}_4 | = 51$. Since $\mathcal{G}\not\subseteq\mathcal{G}_2$, there exists a set $G$ in $\mathcal{G}$ such that $|G\cap A|\le 1, G\cap Y\in \mathcal{A}_2$ or $\mathcal{A}_3$.  w.l.o.g., let $G\cap \{x_1, x_2\}=\emptyset$. Since $ |Y\setminus (A\cup G)|\ge 2 $, we can pick $a, b \in Y\setminus (A\cup G)$ such that $(G\cap Y) \cap \{x_1, x_2, a, b\}=\emptyset$. By the intersecting property of $\mathcal{A}_2\cup\mathcal{A}_3\cup \mathcal{A}_4$, we have $\{x_1, x_2, a, b\}\not\in \mathcal{A}_4$. Hence $ | \mathcal{A}_4 | \le 50$, as desired.

  Suppose that $\mathcal{A}_4\subseteq\mathcal{G}_3$ with core $\{x_1, x_2, x_3\}$ and center $x$. By direct calculation, $ | \mathcal{A}_4 | \le 51$, so we may assume $ | \mathcal{A}_4 | = 51$ and $\mathcal{A}_4=\mathcal{G}_3$.  Since $\mathcal{G}\not\subseteq\mathcal{G}_3$, there must be $G\in \mathcal{G}$ and $G\cap Y\in \mathcal{A}_2  \text{ or }   \mathcal{A}_3$,  such that either $x\not \in G$ and $\{x_1, x_2, x_3\}\not\subseteq G\cap Y$ or $x \in G$ and $\{x_1, x_2, x_3\}\cap (G\cap Y) =\emptyset$.  By the intersecting property of $\mathcal{A}_2\cup \mathcal{A}_3 \cup \mathcal{A}_4$, in either case, we have $\mathcal{A}_4\neq\mathcal{G}_3$ and $ | \mathcal{A}_4 | < 51$.

  At last, suppose that $\mathcal{A}_4\subseteq\mathcal{J}_2$ with center $x$, kernel $\{x_1, x_2, x_3\}$ and the set of pages $\{x_4, x_5\}$. By Theorem 1.4, we may assume $ | \mathcal{A}_4 | = 51$ and $\mathcal{A}_4=\mathcal{J}_2$. Since $\mathcal{A}_2\cup \mathcal{A}_3\cup \mathcal{A}_4$ is intersecting, there is no member in $\mathcal{A}_2  \text{ or }   \mathcal{A}_3$ avoiding $x$. And each member in $\mathcal{A}_2$ must interset with $\{x_1, x_2, x_3\}$, each member in $\mathcal{A}_3$ must interset with $\{x_1, x_2, x_3\}$ or contain $\{x_4, x_5\}$, to satisfy these conditions, $G$ must be contained in $\mathcal{J}_2$, a contradiction.

  So we have proved that $\mathcal{A}_4\le 50$ for $k=4$.

Next consider $k\ge 5$.   Suppose on the contrary that there exists $i\in \{2, \dots, k-1\}$ such that
\begin{equation}\label{eq3}
| \mathcal{A}_i |  > {2k-1 \choose i-1}-{k-1 \choose i-1}-{k-2 \choose i-2}-{k-3 \choose i-3}.
\end{equation}
Note that for $i=2$,
\begin{equation}
{2k-1 \choose i-1} - {k-1\choose i-1} - {k-2\choose i-2} - {k-3 \choose i-3}=k-1\nonumber.
\end{equation}
If $|\mathcal{A}_2|\ge k \,(\, k\ge 5\, )$, then $\mathcal{A}_2$ is EKR, moreover, since $\mathcal{A}_2\cup \mathcal{G}$ is intersecting, $\mathcal{G}$ must be EKR or HM, a contradiction. Hence $|\mathcal{A}_2|\le k-1$, as desired.

Now consider $i\ge 3$. Under the assumption (\ref{eq3}), we claim that
\begin{equation}\label{eq4}
| \mathcal{A}_i | > {2k-1 \choose i-1}-{2k-i-1 \choose i-1}-{2k-i-2 \choose i-2}+2.
\end{equation}
Let us explain inequality (\ref{eq4}). We write
\begin{equation}\label{5}
{2k-i-2 \choose i-2}={2k-i-3 \choose i-2}+{2k-i-3 \choose i-3}.
\end{equation}
For $k\ge 5$ and $3 \le i \le k-1$, we have
\begin{equation}\label{6}
{2k-1-i \choose i-1}-{k-1 \choose i-1}={k-1 \choose i-2}+{k \choose i-2}+\dots+{2k-2-i \choose i-2} \ge4,
\end{equation}
\begin{equation}\label{7}
{2k-i-3 \choose i-2}-{k-2 \choose i-2}\ge 0,\,\,\,
{2k-i-3 \choose i-3}-{k-3 \choose i-3} \ge 0,
\end{equation}
Combining (\ref{eq3}), (\ref{5}), (\ref{6}) and (\ref{7}), we obtain (\ref{eq4}).
Since $\mathcal{A}_i$ is intersecting, we may assume, by Theorem \ref{thm1.4} that
$\mathcal{A}_i$ is EKR or HM or for $i=3$, $\mathcal{A}_i\subseteq\mathcal{G}_2$.

%or $\mathcal{G}_3$ ($i=4$ only). We first consider $i=4$. If $\mathcal{A}_i$ is $\mathcal{G}_2$, then
%\[
%\begin{split}
%| \mathcal{A}_i | & \le 3 {2k-3 \choose 2}+ 2k-3\\
%&\le {2k-1 \choose i-1}-{k-1 \choose i-1}-{k-2 \choose i-2}-{k-3 \choose i-3}
%\end{split}
%\]
%The last inequality holds for $k\ge 5$.
%If $\mathcal{A}_i$ is $\mathcal{G}_3$, then
%\[
%\begin{split}
%| \mathcal{A}_i | & \le {2k-2 \choose 2}+{2k-3 \choose 2}+{2k-4 \choose 2} 2k-4\\
%&\le {2k-1 \choose i-1}-{k-1 \choose i-1}-{k-2 \choose i-2}-{k-3 \choose i-3}
%\end{split}
%\]
%The last inequality also holds for $k\ge 5$.

%%%这段不要

 %Since $\mathcal{G}$ is neither EKR nor HM, there are at least two sets missing $x$ in $\mathcal{G}$.  By Fact \ref{fact0} and inequality (1), we have known that  if there are exactly two sets missing $x$ is $\mathcal{G}$, then
%\[
%|\mathcal{G}| \le {n-1 \choose k-1} -2{n-k-1 \choose k-1} +{n-k-3 \choose k-1}+2.
%\]
%Moreover, by inequality (2), for $i\le {\frac{2k+2}{3}}$, there exists a family $\mathcal{J}_2$ having three sets missing $x$ such that
%\[
%|\mathcal{J}_2| > {n-1 \choose k-1} -2{n-k-1 \choose k-1} +{n-k-3 \choose k-1}+2.
%\]
Case (i): $\mathcal{A}_i$ is EKR or HM at center $x$.

In this case $\mathcal{A}_i$ contains at most 1 $i$-set missing $x$. Recall that there are at least three sets missing $x$ in $\mathcal{G}$. Pick three sets $G_1, G_2, G_3\in \mathcal{G}$ missing $x$. Denote \\
\begin{minipage}[b]{0.5\linewidth}
\begin{flushleft}
\begin{align*}
T&=G_1\cap G_2\cap G_3\cap Y, \,t=|T|,\\
 T_1&=(G_1\cap Y)\setminus (G_2\cup G_3),\, t_1=|T_1|,\\
T_2&=(G_2\cap Y)\setminus (G_1\cup G_3), \,t_2=|T_2|, \\
 T_3&=(G_3\cap Y)\setminus (G_1\cup G_2),\, t_3=|T_3|, \\
T_4&=(G_1\cap G_2\cap Y)\setminus G_3,\, t_4=|T_4|,\\
 T_5&=(G_1\cap G_3\cap Y)\setminus G_2, \,t_5=|T_5|,\\
T_6&=(G_2\cap G_3\cap Y)\setminus G_1,\, t_6=|T_6|.
\end{align*}
\end{flushleft}
\end{minipage}
\hspace{-1cm}
\begin{minipage}[b]{0.5\linewidth}
\begin{flushright}
\includegraphics[scale=0.39]{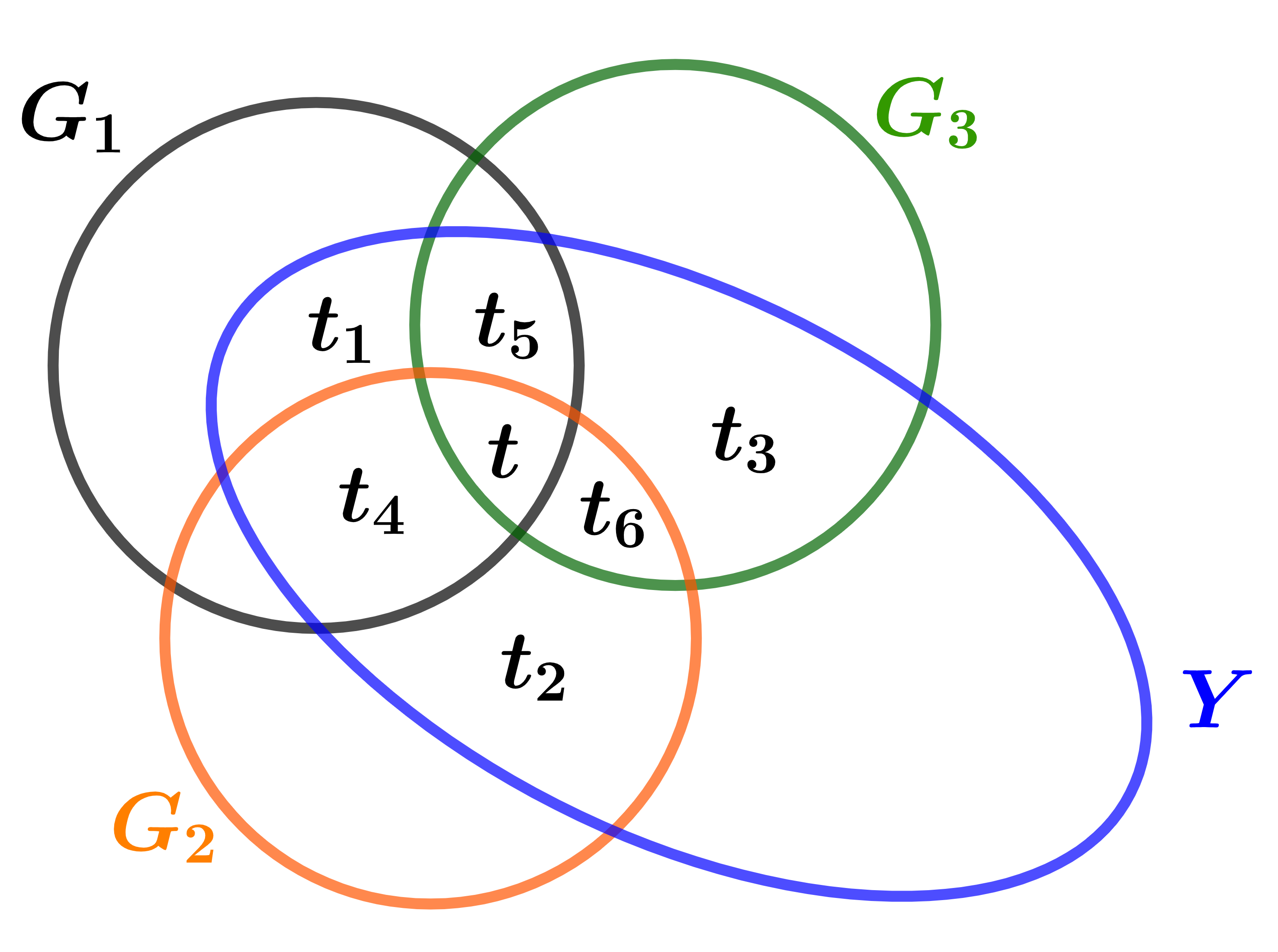}
\end{flushright}
\end{minipage}

Clearly, $t+t_1+t_4+t_5\le k$, $t+t_2+t_4+t_6\le k$, $t+t_3+t_5+t_6\le k$. By Lemma \ref{lem2.2} $\mathcal{A}_i \cup \{G_1\cap Y, G_2\cap Y, G_3\cap Y\}$ is intersecting. Applying Inclusion-Exclusion principle, we have
\begin{equation}\label{eq5}
\begin{split}
 \mathcal{A}_i\le{2k-1 \choose i-1}-{2k-1-t-t_1-t_4-t_5 \choose i-1}
-{2k-1-t-t_2-t_4-t_6 \choose i-1}\\-{2k-1-t-t_3-t_5-t_6 \choose i-1}
+{2k-1-t-t_1-t_2-t_4-t_5-t_6 \choose i-1}\\+{2k-1-t-t_1-t_3-t_4-t_5-t_6 \choose i-1}
+{2k-1-t-t_2-t_3-t_4-t_5-t_6 \choose i-1}\\-{2k-1-t-t_1-t_2-t_3-t_4-t_5-t_6 \choose i-1}+c,
\end{split}
\end{equation}
where $c=0$ (if $\mathcal{A}_i$ is EKR) or $1$ (if $\mathcal{A}_i$ is HM).
Denote the right side of equality (\ref{eq5}) by $f$.
We rewrite it as
\begin{equation}\label{eq6}
\begin{split}
f={2k-1 \choose i-1}-{2k-2-t-t_1-t_4-t_5 \choose i-2}-\dots -{2k-1-t-t_1-t_3-t_4-t_5-t_6 \choose i-2}\\
-{2k-2-t-t_2-t_4-t_6 \choose i-2}-\dots-{2k-1-t-t_1-t_2-t_4-t_5-t_6 \choose i-2}\\
- {2k-2-t-t_3-t_5-t_6 \choose i-2}-\dots-{2k-1-t-t_2-t_3-t_4-t_5-t_6 \choose i-2}\\
-{2k-1-t-t_1-t_2-t_3-t_4-t_5-t_6 \choose i-1}+c.
\end{split}
\end{equation}
We can see that the right side of (\ref{eq6}), consequently (\ref{eq5}) does not decrease as $t+t_1+t_4+t_5, t+t_2+t_4+t_6, t+t_3+t_5+t_6$ increase. Since $t+t_1+t_4+t_5, t+t_2+t_4+t_6, t+t_3+t_5+t_6\le k,$ we can substitute $t+t_1+t_4+t_5=k, t_2+t_4+t_6=k-t, t_3+t_5+t_6=k-t$ into inequality (\ref{eq5}), and this will not decrease $f$. So we have
\begin{equation}\label{eq7}
\begin{split}
|\mathcal{A}_i| &\le
{2k-1\choose i-1}-3{k-1\choose i-1}+{t+t_4-1\choose i-1}+{t+t_5-1\choose i-1}+{t+t_6-1\choose i-1}\\
&-{t+t_5-t_2-1 \choose i-1} +c\\
&={2k-1\choose i-1}-3{k-1\choose i-1}+{t+t_4-1\choose i-1}+{t+t_6-1\choose i-1}+{t+t_5-2\choose i-2}\\
&+\dots+{t+t_5-t_2-1 \choose i-2}+c\\
&\triangleq g.
 \end{split}
\end{equation}
Clearly, $g$ does not decrease as $t+t_4, t+t_5, t+t_6$ increase and $t+t_4\le k-1, t+t_5\le k-1$  $t+t_6\le k-1$. If $t+t_5-t_2-1\ge k-3$, then
\[
\begin{split}
 |\mathcal{A}_i| &\le {2k-1\choose i-1}-3{k-1\choose i-1}+3{k-2\choose i-1}-{k-3 \choose i-1}+c\\
         &={2k-1\choose i-1}-{k-1\choose i-1}-{k-2\choose i-2}-{k-3 \choose i-3}+c.
 \end{split}
\]
The equality holds only if $t=k-1, t_1=t_2=t_3=1, t_4=t_5=t_6=0$.
 If $t+t_5-t_2-1\le k-4$ ($\ast$), then $t\le k-2$ since $t=k-1$ implies $t_5=0$ and combining with ($\ast$), we have $t_2\ge2$, so $t+t_2\ge k+1$, a contradiction. Since $t+t_4\le k-1, t+t_5\le k-1$ and $t+t_6\le k-1$, by (\ref{eq6}) and (\ref{eq7}), taking $t+t_1+t_4+t_5=k, t+t_2+t_4+t_6=k, t+t_3+t_5+t_6=k$ and $t+t_4=k-1, t+t_5=k-1, t+t_6=k-1$ (this implies that $t=k-2, t_4=t_5=t_6=1$ and $t_1=t_2=t_3=0$) does not decrease $f$. So
\[
\begin{split}
g & \le {2k-1\choose i-1}-3{k-1\choose i-1}+3{k-2\choose i-1}-{k-2 \choose i-1}+c\\
&={2k-1\choose i-1}-{k-1\choose i-1}-{k-2\choose i-2}-{k-3 \choose i-3}-{k-3 \choose i-2}+c\\
&\le {2k-1\choose i-1}-{k-1\choose i-1}-{k-2\choose i-2}-{k-3 \choose i-3}-2+c.
\end{split}
\]
So
$$ |\mathcal{A}_i| \le {2k-1\choose i-1}-{k-1\choose i-1}-{k-2\choose i-2}-{k-3 \choose i-3}+c.$$
To reach $c=1$, there is a set $A$ in $\mathcal{A}_i$ not containing $x$. Let $G_1$ be such that $G_1\cap Y=A$. So $|G_1\cap Y|= i\le k-1$. This implies that $t+t_1+t_4+t_5\le k-1$.
%Let $\mathcal{A'}_i$ be the extremal family we get in this condition, then
%\[
%|\mathcal{A}_i|-|\mathcal{A'}_i|\ge -{k-1\choose i-1}+{k\choose i-1}\ge 2.
%\]
In view of (\ref{eq5}) and (\ref{eq6}), $|\mathcal{A}_i|$ strictly decreases as $t+t_1+t_4+t_5$ strictly decreases.
So we have
\[
|\mathcal{A}_i|\le {2k-1\choose i-1}-{k-1\choose i-1}-{k-2\choose i-2}-{k-3 \choose i-3},
\]
as desired.

Case (ii): For $i=3$, $\mathcal{A}_i\subseteq\mathcal{G}_2$ with core, say $\{x_1, x_2, x_3\}$.

By direct calculation, we have $|\mathcal{A}_3|\le 3(2k-3)+1=6k-8$. When $k\ge 5$, we have
$$6k-8< {2k-1\choose 2}-{k-1\choose 2}-{k-2\choose 1}-{k-3\choose 0},$$
 as desired.
\end{proof}

\begin{lemma}\label{lem2.8}
Let $\mathcal{G}$ be the final stable family as in Lemma \ref{lem2.2}. Then
\[
|\mathcal{G}|\le {n-1 \choose k-1} -{n-k-1 \choose k-1} -{n-k-2 \choose k-2} -{n-k-3 \choose k-3}+3.
\]
\end{lemma}

\begin{proof}
Note that for any $A\in\mathcal{A}_i$, there are at most ${n-|Y|}\choose k-i$ $k$-sets in $\mathcal{G}$ containing $A$. For $k=4$, we have
\begin{align}
|\mathcal{G}|\le \sum_{i=1}^{4}|\mathcal{A}_i|{n-9\choose 4-i}.\nonumber
\end{align}
By Lemma \ref{lem2.3},
\begin{align}\label{eq8}
|\mathcal{G}|&\le 3{n-9\choose 2}+18{n-9\choose 1}+50\nonumber\\
&=\frac{3}{2}n^2-\frac{21}{2}n+23\nonumber\\
&={n-1\choose 3}-{n-5\choose 3}-{n-6\choose 2}-{n-7\choose 1}+3.
\end{align}

For $k\ge5$, we have
\begin{align}\label{eq9}
|\mathcal{G}|&\le \sum_{i=1}^{k}|\mathcal{A}_i|{n-2k\choose k-i} \nonumber  \\
&\!\!\!\!\!\! \!\! \overset{\text{Lemma \ref{lem2.3}}}{\le} 3+\sum_{i=1}^{k}
\left( {2k-1\choose i-1}-{k-1\choose i-1}-{k-2\choose i-2}-{k-3\choose i-3} \right){n-2k\choose k-i} \nonumber\\
&={n-1\choose k-1}-{n-k-1\choose k-1}-{n-k-2\choose k-2}-{n-k-3\choose k-3}+3.
\end{align}

\end{proof}

By Lemma \ref{lem2.8}, we have obtained the quantitative part of Theorem \ref{thm1.5}.

%Observe that,
%\begin{eqnarray*}
 %&   &  {n-1 \choose k-1} -2{n-k-1 \choose k-1} +{n-k-3 \choose k-1}+2 \\
 %&> &  {n-1\choose k-1}-{n-k-1\choose k-1}-{n-k-2\choose k-2}-{n-k-3\choose k-3}+3
%\end{eqnarray*}
%if and only if $n<3k-2$. Hence, we have that for $n< 3k-2$,
%\begin{equation}
%|\mathcal{G}|\le  {n-1 \choose k-1} -2{n-k-1 \choose k-1} +{n-k-3 \choose k-1}+2,
%\end{equation}
%and for $n\ge 3k-2$,
%\begin{equation}
%|\mathcal{G}|\le {n-1\choose k-1}-{n-k-1\choose k-1}-{n-k-2\choose k-2}-{n-k-3\choose k-3}+3.
%\end{equation}
%In particular, for $k=4, n=2k+1=3k-3=9$,
%\begin{eqnarray*}
% &&  {n-1 \choose k-1} -2{n-k-1 \choose k-1} +{n-k-3 \choose k-1}+2   \\
% &=&{n-1\choose k-1}-{n-k-1\choose k-1}-{n-k-2\choose k-2}-{n-k-3\choose k-3}+3.
%\end{eqnarray*}

\subsection{Uniqueness Part of Theorem \ref{thm1.5}}
 Let $\mathcal{G}$ be a $k$-uniform family such that the equality holds in Lemma \ref{lem2.8} .We first show the structure of $\mathcal{G}$.

\begin{theorem}\label{thm2.9}
Let $\mathcal{G}$ be a family as in Lemma \ref{lem2.8} such that the equality holds. If $k=5$, then $\mathcal{G}= \mathcal{J}_3$ or $\mathcal{G}_4$; if $k\ne5$, then $\mathcal{G}= \mathcal{J}_3$.\end{theorem}

\begin{proof}
%Suppose $2k<n<3k-2$ first. If $k>4$, by Fact \ref{fact0}, the equality in (20) holds  only if $\mathcal{G}=\mathcal{H}_2$. If $k=4$, $|\mathcal{H}_2|=|\mathcal{J}_3|$, these two families both reach the maximum value.

%Next consider $n\ge 3k-2$. In this case, by Fact \ref{fact0}, $\mathcal{H}_2$ can not reach the maximum value.
 To make the equalities (\ref{eq8}) and (\ref{eq9}) hold, we must  get all the equalities in Lemma \ref{lem2.3}. So $|\mathcal{A}_2|=k-1$. By Lemma \ref{lem2.2}, $\mathcal{A}_2$ is intersecting, so $\mathcal{A}_2$ is a star, say with  center $x$ and leaves $\{x_1, x_2, \dots, x_{k-1}\}$, or a triangle on $\{x, y, z\}$ (only for $k=4$). First consider $k=4$. If $\mathcal{A}_2$ is a triangle, then $\mathcal{G}=\mathcal{G}_2$, a contradiction. Otherwise,
$\mathcal{A}_2$ is a star, this implies that all sets in $\mathcal{G}$ missing $x$ must contain
$\{x_1, x_2, x_3\}$, and the number of such sets is at least 3. Then either $\mathcal{G}= \mathcal{G}_3$ or $\mathcal{G}= \mathcal{J}_i, 3\le i\le k-1$. By the assumption that $\mathcal{G}\not\subseteq \mathcal{G}_3$, the former is impossible, and the latter implies $\mathcal{G}= \mathcal{J}_3$. Hence, the equality in (21) holds only if $\mathcal{G}=\mathcal{J}_3$. For $k\ge5$, $\mathcal{A}_2$ must be a star.
Similarly, in this condition, we have  either $\mathcal{G}=\mathcal{G}_{k-1}$ or $\mathcal{G}= \mathcal{J}_i, 3\le i\le k-1$. In particular, for $k=5$, we can see that the extremal value of $|\mathcal{G}|$ can be achieved by $|\mathcal{G}_4|$ and $|\mathcal{J}_3|$, and for $k>5$, by $|\mathcal{J}_3|$ only.

\end{proof}

We will use some results in \cite{HK2017}.
We say two families $\mathcal{G}$ and $\mathcal{F}$ are \textit{cross-intersecting} if for any $G\in \mathcal{G}$ and $F\in \mathcal{F}$, $G\cap F\neq\emptyset $. We say that a family $\mathcal{F}$ is {\it non-separable} if $\mathcal{F}$ cannot be partitioned into the union of two cross-intersecting non-empty subfamilies.

\begin{proposition}$(\cite{HK2017})$\label{prop3.2}
Let $r\ge 2$. Let $Z$ be a set of size $m\ge2r+1$ and let $A\subseteq Z$ such that  $|A|\in \{r-1, r\}$. Let $\mathcal{B}$ be an $r$-uniform family on $Z$ such that $\mathcal{B}=\{B\subseteq Z: 0<|B\cap A|<|A|\}$. Then $\mathcal{B}$ is non-separable.
\end{proposition}

\begin{lemma}$(\cite{HK2017})$\label{lem3.2}
Let $\mathcal{F}$ be a $k$-uniform intersecting family. If $k\ge 3$ and $S_{xy}(\mathcal{F})\in \{\mathcal{J}_2, \mathcal{G}_{k-1}, \mathcal{G}_2\}$, then  $\mathcal{F}$ is isomorphic to $S_{xy}(\mathcal{F}).$
\end{lemma}

Combining with Theorem \ref{thm2.9} and Lemma \ref{lem3.2}, the uniqueness part  of Theorem \ref{thm1.5} will be completed by showing the following lemma.

\begin{lemma}\label{lem3.3}
Let $\mathcal{F}$ be a $k$-uniform intersecting family. If $k\ge 4$ and $S_{xy}(\mathcal{F})=\mathcal{J}_3$, then  $\mathcal{F}$ is isomorphic to $\mathcal{J}_3$.
\end{lemma}

\begin{proof}
Assume that $S_{xy}(\mathcal{F})=\mathcal{J}_3$ with center $x_0$, kernel $E$ and the set of pages $\{x_1, x_2, x_3\}$. That is
\[
\mathcal{J}_3=\{G:\{x_0, x_1, x_2, x_3\}\subseteq G\}\cup\{G:x_0\in G, G\cap E\ne \emptyset\}\cup \{E\cup\{x_1\}, E\cup\{x_2\}, E\cup\{x_3\}\}.
\]
Define
\begin{align*}
\mathcal{B}_x&:=\{G\in \mathcal{J}_3: x\in G, y\not\in G, (G\setminus{x})\cup{y}\not\in \mathcal{J}_3\}, \\
\mathcal{C}_x&:=\{G\in \mathcal{B}_x: G\in \mathcal{F}\}, \\
\mathcal{D}_x&:=\{G\in \mathcal{B}_x: G\not\in \mathcal{F}\}, \\
\mathcal{B}' &:=\{G\setminus \{x\}: G\in \mathcal{B}_x\},\\
\mathcal{C}'&:=\{G\setminus \{x\}: G\in \mathcal{C}_x\},\\
 \mathcal{D}'&:=\{G\setminus \{x\}: G\in \mathcal{D}_x\}.
\end{align*}
Then $\mathcal{B}_x=\mathcal{C}_x \sqcup \mathcal{D}_x$ and $\mathcal{B}'=\mathcal{C}' \sqcup \mathcal{D}'$.
The definition of $\mathcal{D}_x$ implies that for any $G\in \mathcal{D}_x$, $G\setminus \{x\} \cup \{y\} \in \mathcal{F}$, and the definition of $\mathcal{C}_x$ implies that for any $G\in \mathcal{C}_x$, $G\setminus \{x\} \cup \{y\} \not\in \mathcal{F}$. Clearly, only the sets in $\mathcal{D}_x$ are in  $S_{xy}(\mathcal{F})\setminus \mathcal{F} $. If $\mathcal{D}_x=\emptyset$, then $S_{xy}(\mathcal{F})=\mathcal{F}=\mathcal{J}_3$, and if $\mathcal{C}_x=\emptyset$, then $\mathcal{F}$ is still $\mathcal{J}_3$ with center $y$.  On the other hand, notice that $\mathcal{C}_x$ and $\{ G\setminus \{x\}\cup \{y\}: G\in\mathcal{D}_x\}$ are cross intersecting, so $\mathcal{C}'$ and $\mathcal{D}'$ are cross intersecting. We are going to prove that $\mathcal{B'}$ is non-separable, this means that $\mathcal{C}'=\emptyset$ or $\mathcal{D}'=\emptyset$, and hence $\mathcal{C}_x=\emptyset$ or $\mathcal{D}_x=\emptyset$, we can conclude the proof. So what remains is to show the following claim.

\begin{clm}
$\mathcal{B'}$ is non-separable.
\end{clm}
\begin{proof}
We say the shift $S_{xy}: \mathcal{F} \to  \mathcal{J}_3$ is trivial if $\mathcal{B}_x=\emptyset$. Let $Z:=[n]\setminus\{x, y\}$. If $r=k-1$, then $|Z|\ge 2k+1-2=2r+1$.

Let $T_1:=\{x_0\}, \,T_2:=E, \,T_3:=\{x_1, x_2, x_3\}, \,T_4:=[n]\setminus(T_1\cup T_2\cup T_3)$.

Since for $x, y \in T_i$ or for $x\in T_i, y\in T_j, i>j$, the shift is trivial, we only need to consider the following three cases.

Case (i): $x=x_0$ and $y\in T_2 \cup T_3 \cup T_4$.

If $y\in T_3$, let $A=E$, then $\mathcal{B'}=\{B\subseteq Z: 0<|B\cap A|<|A|\}$. By Proposition \ref{prop3.2}, $\mathcal{B'}$ is non-separable. If $y\in T_2\cup T_4$, let $A:=E\setminus \{y\}$, then $|A|\in \{r-1, r\}$. Assume that $\mathcal{B'}$ has a partition $\mathcal{B'}_1\cup \mathcal{B'}_2$ such that $\mathcal{B'}_1$ and $\mathcal{B'}_2$ are cross-intersecting. We now partition $\mathcal{B'}$ into three parts $\mathcal{P}_1\sqcup \mathcal{P}_2\sqcup \mathcal{P}_3$, where
\begin{gather*}
\mathcal{P}_1:=\{B\subseteq Z: 0<|B\cap A|<|A|\},\\
\mathcal{P}_2:=\{B\in \mathcal{B'}: B\cap A=\emptyset\}=\{T_3\cup F: F\subseteq T_4\setminus \{y\}, |F|=k-4\},
\end{gather*}
and
\begin{eqnarray*}
\mathcal{P}_3:=\{B\in \mathcal{B'}: A\subseteq B\}=
\begin{cases}
\{A\cup \{z\}: z\in T_4\}, & y\in T_2;\\
\{A\}, & y\in T_4.
\end{cases}
\end{eqnarray*}
Obviously, $\mathcal{P}_1\ne \emptyset$. By Proposition \ref{prop3.2}, $\mathcal{P}_1$ is non-separable. For any $P\in \mathcal{P}_2$, and any $a\in A$, we have $|Z\setminus\{a\}|\ge 2r $, then in $\mathcal{P}_1$ we can always find $P'\subseteq Z\setminus (\{a\}\cup P)$  such that $0<|P'\cap A|<|A|$ and $P\cap P'=\emptyset$. This implies that $P$ and $P'$ must be in the same $\mathcal{B'}_i$ ($i=1$ or $2$)(recall that we assumed that $\mathcal{B'}$ has a partition $\mathcal{B'}_1\cup \mathcal{B'}_2$ such that $\mathcal{B'}_1$ and $\mathcal{B'}_2$ are cross-intersecting),  hence $\mathcal{P}_1$ and $\mathcal{P}_2$ are in the same $\mathcal{B'}_i$.
For any $P\in \mathcal{P}_3$, we have $|P\cap T_4|\le 1$. Since $|T_4|\ge k-2$, there is a $(k-4)$-set $F\subseteq T_4\setminus \{y\}$, such that $P\cap F=\emptyset$. Note that $P':=F\cup T_3\in \mathcal{P}_2$ and $P'\cap P=\emptyset$, so $\mathcal{P}_2$ and $\mathcal{P}_3$ are in the same $\mathcal{B'}_i$. Hence $\mathcal{B'}=\mathcal{B'}_1$ or $\mathcal{B'}_2$, as desired.

Case (ii): $x\in T_2$ and $y\in T_3 \cup T_4$.

 Let $E_i:=(E\cup \{x_i\})\setminus \{x\}, i=1, 2, 3$.

 If $y\in T_4$, then
\[
\mathcal{B'}=\{E_1, E_2, E_3\}\cup \left\{G\in{[n]\setminus\{x\} \choose k-1}: x_0\in G, G\cap E=\emptyset, |G\cap T_3|\le 2, y\not\in G\right\}.
\]
Since $|T_4\setminus \{y\}|\ge k-3$, there is $P\in \mathcal{B'}\setminus\{E_1, E_2\}$, such that
%$P\cap \{x_1, x_2\}=\emptyset$, and then
 $P\cap E_1=P\cap E_2=\emptyset$. Hence, $E_1$ and $E_2$ belong to the same part $\mathcal{B'}_i$. Similarly, $E_1$ and $E_3$ belong to the same part. Thus $E_1, E_2$ and $E_3$ are in the same $\mathcal{B'}_i$. Moreover, for any $P'\in  \mathcal{B'}\setminus\{E_1, E_2, E_3\}$, because $|P'\cap \{x_1, x_2, x_3\}|\le 2$, we have $P'\cap E_1=\emptyset$, or $P'\cap E_2=\emptyset$ or $P'\cap E_3=\emptyset$. Hence, $\mathcal{B'}$ is non-separable, as desired.

If $y\in T_3$, w.l.o.g., let $y=x_1$. Then
\[
\mathcal{B'}=\{E_2, E_3\}\cup \left\{G\in{[n]\setminus\{x\} \choose k-1}: x_0\in G, G\cap E=\emptyset, |G\cap T_3|\le 1, y\not\in G\right\}.
\]
Since $|T_4|\ge k-2$, there exists $P\in \mathcal{B'}\setminus \{E_2, E_3\}$ such that $P\cap T_3=\emptyset$, then $P\cap E_2=\emptyset$, and $P\cap E_3=\emptyset$, this implies that $E_2$ and $E_3$ are in the same $\mathcal{B'}_i$. Because $|G\cap T_3|\le 1$ and $G\cap E=\emptyset$, it's not hard to see that each $P\in \mathcal{B'}\setminus \{E_2, E_3\}$ is disjoint from  one of $E_2$ and $E_3$. Hence $\mathcal{B'}$ is non-separable.

Case (iii): $x\in T_3$ and $y\in T_4$. w.l.o.g., let $x=x_1$.

Under this condition,
\[
\mathcal{B'}=\{E\}\cup \left\{G\in{[n]\setminus\{x\} \choose k-1}: \{x_0, x_2, x_3\}\subseteq G, G\cap E=\emptyset, y\not\in G\right\}.
\]
Since $E$ is disjoint from every other set in $\mathcal{B}'\setminus \{E\}$,
 $\mathcal{B'}$ is non-separable.
\end{proof}
The proof of Lemma \ref{lem3.3} is complete.
\end{proof}

\section{Proofs of Lemma \ref{lem2.1} and Lemma \ref{lem2.2}}\label{sec3}

%第3章
\subsection{Proof of Lemma \ref{lem2.1}}

We first show the following preliminary results. For a family $\mathcal{F}\subseteq 2^{[n]}$ and $x_1, x_2, x_3\in [n]$, let $d_{\{x_1, x_2\}}$ be the number of sets containing $\{x_1, x_2\}$ in $\mathcal{F}$, and $d_{\{x_1, x_2, x_3\}}$ be the number of sets containing $\{x_1, x_2, x_3\}$ in $\mathcal{F}$.

\begin{clm} \label{clm 1}
  Let $\mathcal{F}\subseteq\mathcal{G}_2$ be a 4-uniform family with core $A$ satisfying $d_{\{x_1, x_2\}} >2n-7$. Then $\{x_1, x_2\} \subseteq A$.
\end{clm}

\begin{proof}
If $\{x_1, x_2\} \subseteq [n]\setminus A$, then a set in $\mathcal{F}$ containing $\{x_1, x_2\}$ must have two elements from $A$, so $d_{(x_1, x_2)}\leq 3$, a contraction. If $|\{x_1, x_2\} \cap A| = 1$, then  a set in $\mathcal{F}$ containing $\{x_1, x_2\}$ must have at least one element from $A$, so $d_{(x_1, x_2)}\leq 2n-7$, a contraction again. So $\{x_1, x_2\} \subseteq A$, as desired.
\end{proof}

\begin{clm}\label{clm 2}
 Let $\mathcal{F}\subseteq\mathcal{G}_3$ be a 4-uniform family with center $x$ and core $E$ and let $B=\{x\}\cup E$. \\
(i) If $d_{\{x_1, x_2\}}\ge 3n-12$, then $x\in \{x_1, x_2\}$.\\
(ii) If $d_{\{x_1, x_2\}}> 3n-12$, then $\{x_1, x_2\} \subseteq B$ and $x\in \{x_1, x_2\}$.
\end{clm}

\begin{proof}
For (i), assume that $x\not \in \{x_1, x_2\}$. If $\{x_1, x_2\}\cap B= \emptyset$, then the sets containing $\{x_1, x_2\}$ must contain the center $x$ and another vertex from core $E$, so $d_{(x_1, x_2)}\le 3<3n-12$, a contradiction. So $\{x_1, x_2\}\subseteq E$ or $|\{x_1, x_2\} \cap E|=1$. If the former holds, then the sets containing $\{x_1, x_2\}$ must contain the center $x$ or contain the core $E$, so $d_{(x_1, x_2)}\le (n-3)+(n-4)=2n-7<3n-12$, a contradiction. If the latter holds, w.l.o.g., let $\{x_1, x_2\} \cap E=\{x_1\}$, then the sets containing $\{x_1, x_2\}$ must contain the center $x$ or just the set $E\cup \{x_2\} $, so $d_{(x_1, x_2)}\le (n-3)+1<3n-12$, also a contradiction. Hence, $x\in \{x_1, x_2\}$, as desired.

 For (ii), we have shown that $x\in \{x_1, x_2\}$ by (i), w.l.o.g, let $x_1=x$ be the center. If $x_2\not \in E$, then the sets containing $\{x_1, x_2\}$ must intersect with $E$, so $d_{(x_1, x_2)}\le {n-2 \choose 2}-{n-5 \choose 2}=3n-12$, a contradiction to that $d_{\{x_1, x_2\}}> 3n-12$, so $x_2 \in E$,
that is $\{x_1, x_2\} \subseteq B$, as desired.
\end{proof}

\begin{clm}\label{clm00}
 Fix $n>6$. Let $\mathcal{F}\subseteq\mathcal{G}_3$ be a 4-uniform family with center $x$ and core $E$ and let $B=\{x\}\cup E$. If $d_{\{x_1, x_2, x_3\}}\ge n-3$, then either $\{x_1, x_2, x_3\} \subset B$ or $|\{x_1, x_2, x_3\} \cap B|=2$ with $x\in \{ x_1, x_2, x_3\}$.
\end{clm}

\begin{proof}
Suppose on the contrary that neither $\{x_1, x_2, x_3\} \subset B$ nor $|\{x_1, x_2, x_3\} \cap B|=2$ with $x\in \{ x_1, x_2, x_3\}$.
Since $\mathcal{F}\subseteq\mathcal{G}_3$, it's easy to see that if $\{x_1, x_2, x_3\}\subseteq [n]\setminus B$, then $d_{\{x_1, x_2, x_3\}}=0$, so $1\le|\{x_1, x_2, x_3\}\cap B|\le 2$. First consider that $|\{x_1, x_2, x_3\}\cap B|=1$. If $\{x_1, x_2, x_3\}\cap B=\{x\}$, then the sets containing $\{x_1, x_2, x_3\}$ in $\mathcal{F}$ must intersect with $E$, so $d_{\{x_1, x_2, x_3\}}\le3<n-3$, a contradiction. If $|\{x_1, x_2, x_3\}\cap E|=1$, then the set containing $\{x_1, x_2, x_3\}$ in $\mathcal{F}$ must contain $x$, so $d_{\{x_1, x_2, x_3\}}\le1<n-3$, also a contradiction. Hence $|\{x_1, x_2, x_3\}\cap B|=2$. By hypothesis, $|\{x_1, x_2, x_3\}\cap E|=2$, w.l.o.g., let $\{x_1, x_2, x_3\}\cap E=\{x_1, x_2\}$, then $d_{\{x_1, x_2, x_3\}}\le2$ since the possible sets in $\mathcal{F}$ containing $\{x_1, x_2, x_3\}$ are $\{x_1, x_2, x_3\}\cup \{x\}$ and $E\cup \{x_3\}$, a contradiction.
\end{proof}

{\it Proof of Lemma \ref{lem2.1}}.
We first consider that $k\ge 5$.

In {\it Case 1}, i.e., $S_{xy}(\mathcal{H}_1)$ is EKR with center $x$, we  take $X_1=\{x, y\}$. In {\it Case 2}, since $S_{xy}(\mathcal{H}_2)$ is HM at center $x$, let $E=\{z_1, z_2, \dots ,z_k\}$ be the only member  missing $x$, and without loss of generality, we assume $z_1\ne y$, and take $X_2=\{x, y, z_1\}$. In {\it Case 3}, $S_{xy}(\mathcal{H}_3)\subseteq\mathcal{J}_2$ with center $x$, kernal $\{z_1, z_2, \dots ,z_{k-1}\}$. Without loss of generality, we assume $z_1\ne y$, and take $X_3=\{x, y, z_1\}$. We can see that for any set $G\in \mathcal{H}_i$, $G \cap X_i \ne \emptyset$, for $i=1,2,3$. After the shifts $S_{x'y'}$ for all $x'<y', x', y' \in  [n]\setminus X_i$ to $\mathcal{H}_i$, the resulting family $\mathcal{H'}_i$ satisfies that for every set $G' \in \mathcal{H'}_i$, $G'\cap X_i\ne \emptyset$. By the maximality of $|\mathcal{H}|$, we may assume that all $k$-sets containing $X_i  \,(i=1, 2, 3)$ are in $\mathcal{H}$, so is in $\mathcal{H}_i$. These sets will keep stable after any shift $S_{x'y'}$, so there are at least ${n-3 \choose k-2}$ (or ${n-4 \choose k-3}$) $>2$ sets  missing $x'$ in $\mathcal{H}'_i$.  Fact \ref{fact 2} (i), (ii) and (iii) implies that $\mathcal{H}'_i$ is neither EKR nor HM nor contained in $\mathcal{J}_2$.  We are done for $k\ge 5$.

We now assume that $k=4$. We will complete the proof by showing the following Lemmas corresponding to Cases 1-5 in Remark \ref{remark2.4}

\begin{lemma}[{\it Case 1}]\label{case1}
If we each a 4-uniform family $\mathcal{H}_1$ such that $S_{xy}(\mathcal{H}_1)$ is EKR at $x$, then there is a set $X_1=\{x, y, y', z, w\}$ such that after a series of shifts $S_{x'y'}$ $(x'<y'$ and $x', y'\in [n]\setminus X_1)$ to $\mathcal{H}_1$, we will reach a stable family $\mathcal{G}$ satisfying the conditions of Theorem \ref{thm1.5}. Moreover, $\{y, y', z, w\}$ or $ \{x, y', z, w\}$ is in  $\mathcal{G}$. Furthermore,   $G\cap \{x, y\}\ne \emptyset$ for any $G\in \mathcal{G}$.
\end{lemma}

\begin{proof}
Since $S_{xy}( \mathcal{H}_1)$ is EKR, for any $F\in \mathcal{H}_1$, we have $F \cap \{x, y\} \ne \emptyset$. Any set obtained by performing shifts $[n]\setminus \{x, y\}$ to a set in $\mathcal{H}_1$ still contains $x$ or $y$.
We will show Claims \ref{clm 3}, \ref{clm 4} and \ref{clm 6} implying Lemma \ref{case1}.
\begin{clm}\label{clm 3}
Performing shifts in $[n]\setminus \{x, y\}$ to $\mathcal{H}_1$ repeatedly will not reach {\it Cases 1-3} in Remark \ref{remark2.4}.
\end{clm}
\begin{proof}
Since $S_{xy}( \mathcal{H}_1)$ is EKR, for any $G\in \mathcal{H}_1$, we have $G \cap \{x, y\} \ne \emptyset$. By the maximality of $|\mathcal{H}| $ ($|\mathcal{H}_1|$ as well), we have
\begin{align}\label{eq10}
&\left\{G\in {[n]\choose k}: \{x, y\}\subseteq G \right\}\subseteq \mathcal{H}_1,\nonumber\\
&\vert \left\{ G\in \mathcal{H}_1: \{x, y\}\subseteq G \right\}\vert={n-2\choose 2}.
\end{align}
All these sets containing $\{x, y\}$ are stable after performing $S_{x'y'}\,\,(x'<y', x', y'\not \in \{x, y\}) $. So there are still at least ${n-3 \choose 2}>2$ sets missing $x'$ after $S_{x'y'}$, so  we will not reach {\it Case 1-3}.
\end{proof}

\begin{clm}\label{clm 4}
If performing some shifts in $[n]\setminus \{x, y\}$ repeatedly to $\mathcal{H}_1$ reaches $\mathcal{H}_4$ in {\it Case 4}( $S_{x'y'}(\mathcal{H}_4)\subseteq\mathcal{G}_2$), then there exists $X_1=\{x, y, y', z, w\}$ such that performing shifts in $[n]\setminus X_1$ repeatedly to $\mathcal{H}_4$ will not reach {\it Cases 1-5} as in Remark \ref{remark2.4}, and $\{y, y', z, w\}$ or $ \{x, y', z, w\}$ is in the final stable family $\mathcal{G}$.
\end{clm}

\begin{proof}
Assume that after some shifts in $[n]\setminus \{x, y\}$ to $\mathcal{H}_1$, we get $\mathcal{H}_4$ such that $S_{x'y'}(\mathcal{H}_4)\subseteq\mathcal{G}_2$ with core $A$. Since there are ${n-2 \choose 2}$ sets containing $\{x, y\}$ in
$\mathcal{H}_1$ and they are stable (so in $\mathcal{H}_4$), and ${n-2 \choose 2} > 2n-7$ ($n\ge6$), by Fact \ref{fact 2} (ii) and Claim \ref{clm 1}, $A=\{x', x, y\}$. Since $S_{x'y'}(\mathcal{H}_4)\subseteq\mathcal{G}_2$ with core $\{x', x, y\}$, there exists $\{y, y', z_1, w_1\}$ (or $\{x, y', z_2, w_2\}$) in $\mathcal{H}_4$. Let $ X_1:=\{x, y, y', z_1, w_1\}$ (or $X_1:=\{x, y, y', z_2, w_2\}$). Clearly, any set containing $\{x, y\}$ and missing $x''\in [n]\setminus X_1$  are stable after performing shifts in $[n]\setminus X_1$ repeatedly to $\mathcal{H}_4$, so
performing shifts $S_{x''y''}, x'', y'' \in [n]\setminus X_1$ to $\mathcal{H}_4$ will not reach {\it Cases 1-3}. If we reach {\it Case 4}, that is we get a family $\mathcal{H'}_4$, such that $S_{x''y''}(\mathcal{H'}_4)\subseteq\mathcal{G}_2$ with core $A'$. By Fact \ref{fact 2} and Claim \ref{clm 1}, we have $A'=\{x'', x, y\}$. However,
$\{y, y', z_1, w_1\}$ (or $\{x, y', z_2, w_2\}$) is stable under all the shifts in $[n]\setminus X_1$,  so it is still in $S_{x''y''}(\mathcal{H'}_4)$, contradicting that $S_{x''y''}(\mathcal{H'}_4)\subseteq\mathcal{G}_2$ with core $\{x'', x, y\}$. Thus we can not reach {\it Case 4}.

Now assume that after some shifts in $[n]\setminus X_1$ to $\mathcal{H}_4$, we get $\mathcal{H}_5$ such that $S_{x''y''}(\mathcal{H}_5)\subseteq\mathcal{G}_3$ with center and core forming a 4-set $B$ for some $x''$ and $y''\in [n]\setminus X_1$. By Fact \ref{fact 2} (iv), (\ref{eq10}) and Claim \ref{clm 2} (ii), we have $\{x, y, x''\}\subseteq B$. Since there are ${n-2 \choose 2}$ sets which contain $\{x, y\}$ in $\mathcal{H}_1$ (so in $S_{x''y''}(\mathcal{H}_5)$), we have one of the following cases:

($\ast$) $x$ is the center, and $y$ is in the core;

($\ast \ast$) $y$ is the center, and $x$ is in the core.

Recall that there exists $\{y, y', z_1, w_1\}$ or $\{x, y', z_2, w_2\}$ in $\mathcal{H}_4$. We will meet one of the following three cases:

(a) There is no set $G\in \mathcal{H}_4$  such that $G \cap \{x, y, x'\}=\{x\}$. So there exists $\{y, y', z_1, w_1\}\in \mathcal{H}_4$, and all sets containing $\{x', x\}$ in $S_{x'y'}(\mathcal{H}_4)$ are originally in $\mathcal{H}_4$. Take $X_1:=\{x, y, y', z_1, w_1\}$. By the maximality of $|\mathcal{H}|$ (so is $|\mathcal{H}_4|$), there are ${n-2 \choose 2}$ sets containing $\{x', x\}$ in $\mathcal{H}_4$ (so in $S_{x''y''}(\mathcal{H}_5)$ as well). This implies that $x'\in E$, and $x$ is the center. However, $\{y, y', z_1, w_1\}$ is contained in $S_{x''y''}(\mathcal{H}_5)$, a contraction to that $S_{x''y''}(\mathcal{H}_5)\subseteq\mathcal{G}_3$ with center $x$ and core $\{y, x', x''\}$.

(b) There is no set $G\in \mathcal{H}_4$ such that $G \cap \{x, y, x'\}=\{y\}$. So there exists $\{x, y', z_2, w_2\}\in \mathcal{H}_4$, and all sets containing $\{x', y\}$ in $S_{x'y'}(\mathcal{H}_4)$ are originally in $\mathcal{H}_4$. Take $X_1:=\{x, y, y', z_2, w_2\}$. By the maximality of $|\mathcal{H}|$ (so is $|\mathcal{H}_4|$), there are ${n-2 \choose 2}$ sets containing $\{x', y\}$ in $\mathcal{H}_4$, so in $S_{x''y''}(\mathcal{H}_5)$. This implies that $x'\in E$ and $y$ is the center for $S_{x''y''}(\mathcal{H}_5)$. However, $\{x, y', z_2, w_2\}$ is in $S_{x''y''}(\mathcal{H}_5)$, contradicting to that $S_{x''y''}(\mathcal{H}_5)\subseteq\mathcal{G}_3$ at center $y$ and core $\{x, x', x''\}$.

(c) There are both $\{y, y', z_1, w_1\}$ and $\{x, y', z_2, w_2\}$ in $\mathcal{H}_4$. We choose $X_1:=\{x, y, y', z_1, w_1\}$ first. Assume that ($\ast$) happens. Since $\{y, y', z_1, w_1\}$ is still in $S_{x''y''}(\mathcal{H}_5)$, this contradicts that $S_{x''y''}(\mathcal{H}_5)\subseteq\mathcal{G}_3$ with  center $x$ and $\{y, x''\}$ contained in the core. So we assume that ($\ast \ast$) happens. Let $B=\{x, y, x'', u\}$ for some $u$. If $u=x'$, then the existence of $\{y, y', z_1, w_1\}$ makes a contradiction again. Now consider $u\ne x'$.

\begin{clm} \label{clm 5}
If $u\ne x'$, then $u = y'$.
\end{clm}
\begin{proof}
Assume on the contrary that $u\ne y'$. We have shown that $S_{x''y''}(\mathcal{H}_5)$ can not be contained in $\mathcal{J}_2$ at center $y$, then there are  at least 3 sets containing $\{x, u, x''\}$. Although $\{x, x', x'', u\}$ and $\{x, y', x'', u\}$ may be two such sets, there must be $\{x, u, x'', v\} \in S_{x''y''}(\mathcal{H}_5)$ for some $v\in [n]\setminus \{x, y, u, x', y', x''\}$. However, every set in $\mathcal{H}_4$  contains $\{x, y\}$, or $\{x', x\}$, or $\{x', y\}$, or $\{x, y'\}$, or $\{y, y'\}$ by recalling that $S_{x'y'}(\mathcal{H}_4)\subseteq\mathcal{G}_2$ with core $\{x, y, x'\}$, so is every set in $S_{x''y''}(\mathcal{H}_5)$ since $x'', y''\in [n]\setminus \{x, y, y', z_1, w_1\}$, a contradiction.
\end{proof}
By Claim \ref{clm 5}, we have that $S_{x''y''}(\mathcal{H}_5)\subseteq\mathcal{G}_3$ at center $y$ and core $\{x, x'', y'\}$. This time,  we change $X_1$ to $X'_1:=\{x, y, y', z_2, w_2\}$. Similar to the lines in the first paragraph of the proof of Claim \ref{clm 4}, we will not reach {\it Cases 1-4} after performing shifts $S_{x'y'}$ in $[n]\setminus X'_1$. If we reach {\it Case 5}, that is, after some shifts in $[n]\setminus X'_1$ to $\mathcal{H}_4$, we get $\mathcal{H'}_5$ such that $S_{x'''y'''}(\mathcal{H'}_5)\subseteq\mathcal{G}_3$ with center and core forming a 4-set $B'$ for some $x''', y'''\in [n]\setminus X'_1$. By the previous analysis, $B'=\{x, y, x''', y'\}$, and we only need to consider the case that $x$ is the center (If $y$ is the center,  since $\{y, y', z_2, w_2\}$ is still in $S_{x'''y'''}(\mathcal{H'}_5)$, this contradicts that $S_{x'''y'''}(\mathcal{H'}_5)\subseteq\mathcal{G}_3$ with  center $y$ and core $\{x, y', x'''\}$). We have shown that $S_{x''y''}(\mathcal{H}_5)$ can not be contained in $\mathcal{G}_2$ with core $\{x, y, y'\}$, so there is $G\in S_{x''y''}(\mathcal{H}_5)$ such that $G\cap \{x, y\}=\emptyset $ or $G\cap \{x, y'\}=\emptyset$ or $G\cap \{y, y'\}=\emptyset$. Since $S_{x''y''}(\mathcal{H}_5)\subseteq\mathcal{G}_3$ with core $\{x, x'', y' \}$ and center $y$, $G$ must contain $x$ or $y$. If $G\cap \{y, y'\}=\emptyset$, it contradicts that $S_{x''y''}(\mathcal{H}_5)\subseteq\mathcal{G}_3$ with core $\{x, x'', y' \}$ and center $y$. So there is $G\in S_{x''y''}(\mathcal{H}_5)$ such that $G\cap \{x, y'\}=\emptyset$. After shifts in $[n]\setminus X'_1$ to $G$, we get $G'$ missing $x$ and $y'$ still. This contradicts that $S_{x'''y'''}(\mathcal{H'}_5)\subseteq\mathcal{G}_3$ with core $\{ y, x''', y'\}$ and center $x$. Hence, we will not reach {\it Case 5}.

In summary, we have shown that there exists $X_1$ in the form of $\{x, y, y', z, w\}$ such that performing shifts in $[n]\setminus X_1$ repeatedly to $\mathcal{H}_4$ will not reach {\it Cases 1-5} as in Remark \ref{remark2.4}. Moreover, $\{y, y', z, w\}$ or $ \{x, y', z, w\}$ is in the final stable family $\mathcal{G}$. This completes the proof of Claim \ref{clm 4}.
\end{proof}

\begin{clm}\label{clm 6}
If performing some shifts in $[n]\setminus \{x, y\}$ repeatedly to $\mathcal{H}_1$ does not reach {\it Cases1-4}, but reaches $\mathcal{H}_5$ in {\it Case 5} ($S_{x'y'}(\mathcal{H}_5)\subseteq\mathcal{G}_3$), then there exists $X_1$ in the form of $\{x, y, y', z, w\}$ such that performing shifts in $[n]\setminus X_1$ repeatedly to $\mathcal{H}_4$ will not reach {\it Cases 1-5} as in Remark \ref{remark2.4}. Moreover, $\{y, y', z, w\}$ or $ \{x, y', z, w\}$ is in the final stable family $\mathcal{G}$.
\end{clm}

\begin{proof}
Suppose that we get some $\mathcal{H}_5$ such that $S_{x'y'}(\mathcal{H}_5)\subseteq\mathcal{G}_3$ with center and core forming a 4-set $B$. By (\ref{eq10}) and Claim \ref{clm 2}, the center must be $x$ or $y$, and $\{x, y\}\subset B$.  By Fact \ref{fact 2} (iv), $X'\in B$ and $y'\not \in B$. Let $B=\{x, y, x', z\}$. We consider the case that $x$ is the center, the proof for $y$ being the center is similar.

Since $S_{x'y'}(\mathcal{H}_5)\subseteq\mathcal{G}_3$, and recall that we are under Case 1, every set in $\mathcal{H}_5$ intersects  $\{x, y\}$, there exists $\{y, y' , z, w\}$ (or $\{x, y' , z_1, z_2\}$) $\in \mathcal{H}_5$. And by the maximality of $|\mathcal{H}|$ (so is $|\mathcal{H}_5|$), we may assume that all the sets containing $\{x, z\}$ in $S_{x'y'}(\mathcal{H}_5)$ are originally in  $\mathcal{H}_5$.  Let $X_1:=\{x, y, y', z, w\}$ (or $\{x, y, y' , z_1, z_2\}$). Similar to the analysis in the first paragraph of the proof of  Claim \ref{clm 4}, for any shifts $S_{x''y''}$ to $\mathcal{H}_5$  in $[n]\setminus X_1$, we won't reach {\it Cases 1-4}. If we reach {\it Case 5} again, then the resulting family $S_{x''y''}(\mathcal{H'}_5)$ ($x''$ and $y'' \in [n]\setminus X_1$) must be contained in $\mathcal{G}_3$ with core $\{y, x'', z\}$ and center $x$. However $\{y, y' , z, w\}  \, (\text{or} \,\{x, y' , z_1, z_2\} )$ is still in $S_{x''y''}(\mathcal{H'}_5)$, and misses $x''$ and $x$ (or $\{x'', z, y\}\cap \{x, y', z_1, z_2\}=\emptyset$), contradicting that the family $S_{x''y''}(\mathcal{H'}_5)\subseteq\mathcal{G}_3$ with core $\{y, x'', z\}$ and center $x$. So we will not achieve {\it Case 5}, as desired.
\end{proof}

By Claims \ref{clm 3}, \ref{clm 4} and \ref{clm 6}, we have shown that if we reach a $4$-uniform family $\mathcal{H}_1$ such that $\mathcal{H}_1$ is EKR, then there exists a set $X_1$ with $|X_1|\le 5 $ and $\{x, y\}\subseteq X_1$ such that performing shifts $S_{x'y'}$ in $[n]\setminus X_1$ repeatedly to $\mathcal{H}_1$ will result in a stable family satisfying the conditions of Lemma \ref{case1}. This completes the proof of Lemma \ref{case1}.
\end{proof}

\begin{lemma}[{\it Case 2}]\label{case2}
If we each a 4-uniform family $\mathcal{H}_2$ such that $S_{xy}(\mathcal{H}_2)$ is HM at $x$, then there is a set $X_2=\{x, y, z_1, z_2, z_3\}$ such that after a series of shifts $S_{x'y'}$ $(x'<y'$ and $x', y'\in [n]\setminus X_2)$ to $\mathcal{H}_2$, we will reach a stable family $\mathcal{G}$ satisfying the conditions of Theorem \ref{thm1.5}. Moreover, $\{z_1, z_2, z_3, y\}$ or $\{z_1, z_2, z_3, z'_4\} $ $\in \mathcal{G}$. Furthermore, if $\{z_1, z_2, z_3, y\}\in \mathcal{G}$, then every member in $\mathcal{G}$ contains $x$ or $y$. If $\{z_1, z_2, z_3, z'_4\}\in \mathcal{G}$, then every other member in $\mathcal{G}$ contains $x$ or $y$.
\end{lemma}
\begin{proof}
 Note that $S_{xy}(\mathcal{H}_2)$ contains exactly one set, say, $G_0=\{z_1, z_2, z_3, z_4\}$, that misses $x$. W.l.o.g., let $z_1, z_2, z_3\ne y$. Let $X_2:=\{x, y, z_1, z_2, z_3\}$. By the maximality of $|\mathcal{H}_2|$, we may assume
$$\left\{ G\in {X \choose 4}: \{x, y\} \subseteq G, G\cap G_0 \ne \emptyset
\right\}\subseteq \mathcal{H}_2.$$
If $y\in G_0$, that is, $ y=z_4$, then
\begin{align}\label{eq11}
\vert\{G\in \mathcal{H}_2: \{x, y\}\}\vert={n-2\choose 2},
\end{align}
Otherwise, $y\not \in G_0$. We have
\begin{align}\label{eq12}
\vert\{G\in \mathcal{H}_2: \{x, y\}\}\vert=4n-18.
\end{align}

In particular, $\{x, y , z_1, z_2\}$, $\{x, y , z_1, z_3\}$ and $\{x, y , z_2, z_3\}$ are in $\mathcal{H}_2$. Assume that applying shifts in $[n]\setminus X_2$ to $\mathcal{H}_2$, we get $\mathcal{H'}$, such that $S_{x'y'}(\mathcal{H'})$ is EKR or HM or contained in $\mathcal{J}_2$ at center $x'$. However, the three sets $\{x, y , z_1, z_2\}$, $\{x, y , z_1, z_3\}$ and $\{x, y , z_2, z_3\}$ are still in $S_{x'y'}(\mathcal{H'})$ and they miss $x'$, a contradiction. Thus we will not reach {\it Cases 1-3}.

 Assume we reach {\it Case 4} as in Remark \ref{remark2.4}, i.e., $S_{x'y'}(\mathcal{H'})\subseteq\mathcal{G}_2$  with core $A$. By (\ref{eq11}), (\ref{eq12}), Claim \ref{clm 1} and Fact \ref{fact 2} (ii),  we have $A=\{x, y, x'\}$. However $\{z_1, z_2, z_3\}\cap \{x, y , x', y'\} = \emptyset$, after a series of shifts of $[n]\setminus X_2$ to $G_0=\{z_1, z_2, z_3, z_4\}$, we get the resulting set $G'_0\in \mathcal{H'}$ satisfying that $|G'_0\cap |\le \{x, y , x', y'\}1$, a contradiction to that $S_{x'y'}(\mathcal{H'})\subseteq\mathcal{G}_2$ with core $\{x, y, x'\}$. Thus we will not reach {\it Case 4}.

At last, assume $S_{x'y'}(\mathcal{H'})\subseteq\mathcal{G}_3$ as in Remark \ref{remark2.4} ({\it Case 5}) with center and core forming a 4-set $B$. By Fact \ref{fact 2} (iv), $x'\in B$. By Claim \ref{clm 2} (ii) and (\ref{eq11}), (\ref{eq12}), there are at least $4n-18> 3n-12$ ( $n > 6$ ) sets containing $\{x, y\}$, so $\{x, y, x'\}\subset B$. And if $\{x, y\}\subset E$, then the number of sets containing $\{x, y\}$ in $\mathcal{H'}$ is at most $2n-7$, which is smaller than $4n-18$, this contradicts to (\ref{eq12}). Thus the resulting family can only have center $x$ or center $y$. First assume $y\in G_0$, that is $y=z_4$ and $G_0=\{y, z_1, z_2, z_3\}$. This implies that $\{x, z_1, z_2, z_3\}\in \mathcal{H}_2$. Both $\{x, z_1, z_2, z_3\}$ and $G_0=\{z_1, z_2, z_3, z_4\}$ are stable under shifts $S_{x'y'}$ $(x'<y'$ and $x', y'\in [n]\setminus X_2)$, so both of them are in $S_{x'y'}(\mathcal{H'})$. Since $x, x' \not \in G_0$ and $S_{x'y'}(\mathcal{H'})\subseteq\mathcal{G}_3$ with $B \supset \{x, y, x'\}$, $x$ can not be the center. But if $y$ is the center, since $x', y \not \in \{x, z_1, z_2, z_3\}$, also a contradiction. Next assume $y \not \in G_0$.
Notice that $\{z_1, z_2, z_3\} \cap \{x, y, x', y'\} = \emptyset$, after a series of shifts of $[n]\setminus X_2$ to $G_0$, the resulting set $G'_0\in S_{x'y'}(\mathcal{H'})$ satisfies that $G'_0 \cap \{x, y\}= \emptyset$, also contradicts that $S_{x'y'}(\mathcal{H'})\subseteq\mathcal{G}_3$ with $B\supset\{x, y, x'\}$, hence we will not reach {\it Case 5}.

Notice that if $y\in G_0$, we have $\{x, z_1, z_2, z_3\}\in \mathcal{H}_2$ and $G_0=\{y, z_1, z_2, z_3\}\in \mathcal{H}_2$. Note that $\{z_1, z_2, z_3, y\}$ is stable under shifts $S_{x'y'}$ $(x'<y'$ and $x', y'\in [n]\setminus X_2)$,  so $G_0=\{z_1, z_2, z_3, y\}\in \mathcal{G}$. In this case, every member in $\mathcal{H}_2$ contains $x$ or $y$,  Since every member in $\mathcal{H}_2$ is stable at $x$ and $y$,   every member in $\mathcal{G}$ contains $x$ or $y$.
 If $y\not \in G_0$, then $G'_0=\{z_1, z_2, z_3, z'_4\}\in \mathcal{G}$ for some $z'_4\neq y$, and this is the only set in $\mathcal{G}$ that disjoint from set $\{x, y\}$.
\end{proof}

\begin{lemma}[{\it Case 3}]\label{case3}
If we each a 4-uniform $\mathcal{H}_3$ such that $S_{xy}(\mathcal{H}_3)\subseteq\mathcal{J}_2$ at center $x$, kernel $E$ and the set of pages $J$, then there is a set $X_3=\{x, y, z_1, z_2, z_3\}$ such that after a series of shifts $S_{x'y'}$ $(x'<y'$ and $x', y'\in [n]\setminus X_3)$ to $\mathcal{H}_3$, we will reach a stable family $\mathcal{G}$ satisfying the conditions of Theorem \ref{thm1.5} and $G\cap X_3\ne \emptyset$ for any $G\in \mathcal{G}$. Moreover,  either $|G\cap X_3|\ge 2$ for any  $G\in \mathcal{G}$, or $|G\cap G'|\ge 2$ if $G\cap X_3=\{x\}$ and $G'\cap X_3=\{y\}$.
\end{lemma}

\begin{proof}
We will meet one of the following three cases. Case (a): $y\in E$. In this case, let $E=\{y, z_1, z_2\}, J=\{z_3, z_4\}$ and $X_3:=\{x, y, z_1, z_2, z_3\}$. Case (b): $y\in J$. In this case, let $E=\{z_1, z_2, z_3\}, J=\{y, z_4\}$ and $X_3:=\{x, y, z_1, z_2, z_3\}$. Case (c): $y\in [n]\setminus (E\cup J\cup \{x\})$. In this case, let $E=\{z_1, z_2, z_3\}, J=\{z_4, z_5\}$ and  $X_3:=\{x, y, z_1, z_2, z_3\}$.

In each of the above three cases, by the maximality of $|\mathcal{H}|$ $(|\mathcal{H}_3|$ as well), $\{x, y, z_1, z_2\}$, $\{x, y, z_1, z_3\}$, $\{x, y, z_2, z_3\}$ are in $\mathcal{H}_3$, and  they are stable after a series of shifts in $[n]\setminus X_3$, so we will not reach {\it Cases 1-3} after performing shifts in $[n]\setminus X_3$. Assume that applying shifts in $[n]\setminus X_3$ to $\mathcal{H}_3$, we get $\mathcal{H''}$, such that $\mathcal{H'}:=S_{x'y'}(\mathcal{H''})\subseteq\mathcal{G}_2$ with core $A$. Similarly, by the maximality of $|\mathcal{H}_3|$ and direct computation, we have the following claim:
\begin{clm}\label{clm3.7}
There are at least ${n-2 \choose 2}, 4n-18, 3n-11$ members that contain $\{x, y\}$ in Cases (a), (b), (c) respectively.
\end{clm}
 Notice that ${n-2 \choose 2}, 4n-18, 3n-11> 2n-7$. By Claim \ref{clm 1}, Claim \ref{clm3.7} and Fact \ref{fact 2} (ii), $A=\{x, y, x'\}$. In Case (a) or (b), we can see that $\{y,  z_1, z_2, z_3\}\in \mathcal{H'}, |\{y,  z_1, z_2, z_3\}\cap A|=1$, a contradiction. In Case (c), we have  $\{z_1, z_2, z_3, z_4\}\in \mathcal{H}_3$, after some shifts in $[n]\setminus X_3$, it becomes $F$ in $\mathcal{H'}$, and $ |F\cap A|\le1$, a contradiction to that $\mathcal{H'}\subseteq\mathcal{G}_2$ with core $\{x, y, x'\}$. Thus we will not reach {\it Case 4} after performing shifts in $[n]\setminus X_3$ repeatedly.

At last, we assume that $\mathcal{H'}:=S_{x'y'}(\mathcal{H''})\subseteq \mathcal{G}_3$ with center and core forming a 4-set $B$. By Claim \ref{clm 2}, Claim \ref{clm3.7} and Fact \ref{fact 2} (iv), we have $\{x, y, x'\}\subseteq B$, and the center of $\mathcal{H'}$ must be $x$ or $y$. In Cases (a) and (b), we have $\{y, z_1, z_2, z_3\} \in \mathcal{H}_3$, so in $\mathcal{H'}$. Since $ x, x'\not \in \{y, z_1, z_2, z_3\}$, $\mathcal{H'}$ can not be contained in $\mathcal{G}_3$ with $B\supset \{x, y, x'\}$ and center $x$. Since $\{x, z_1, z_2, z_3\} \in \mathcal{H}_3$, so in $\mathcal{H'}$ as well. Notice that $ y, x'\not \in \{x, z_1, z_2, z_3\}$,  $\mathcal{H'}$ can not be contained in $\mathcal{G}_3$ with $B\supset \{x, y, x'\}$ and center $y$. A contradiction. Now consider Case (c). In this case, $\{z_1, z_2, z_3, z_4\}\in \mathcal{H}_3$. Because it is stable at $\{z_1, z_2, z_3\}$ under any shift in $[n]\setminus X_3$, the resulting set $\{z_1, z_2, z_3, z'_4\}$ does not contain  $x$ or $y$. This contradicts that $\mathcal{H'}\subseteq\mathcal{G}_3$ with $B\supset \{x, y, x'\}$ and center $x$ or $y$.

 If Case (a) or (b) happens, then  any 4-set $G\in \mathcal{G}$ satisfies $|G\cap X_3|\ge 2$. If Case (c) happens,
since $\{z_1, z_2, z_3, z_4\}$ and $\{z_1, z_2, z_3, z_5\}$ are the only two sets  disjoint from $\{x, y\}$ in $S_{xy}(\mathcal{H}_3)$, then every set in $\mathcal{H}_3$ (so in $\mathcal{G}$) missing $x$ and $y$ must contain $\{z_1, z_2, z_3\}$. If $x\in G$, $y\in G'$ and $G\cap \{z_1, z_2, z_3, y\}=G'\cap \{z_1, z_2, z_3, x\}=\emptyset$, let $F, F'$ $\in \mathcal{H}_3$ such that $G$ and $G'$ become their resulting sets in $\mathcal{G}$ after a series of shifts in $[n]\setminus X_3$. By the reason that $S_{xy}(\mathcal{H}_3)\subseteq \mathcal{J}_2$ with center $x$, kernel $\{z_1, z_2, z_3\}$ and the set of pages $\{z_4, z_5\}$, for any set $H\in\mathcal{H}_3$ satisfying that $|H\cap \{x, y\}|=1$ and $H\cap\{z_1, z_2, z_3\}=\emptyset$, we have $\{z_4, z_5\}\subseteq H$. So $\{z_4, z_5\}\subseteq F\cap F'$, consequently, $|G\cap G'|\ge 2$.
\end{proof}

\begin{lemma}[{\it Case 4}]\label{case4}
If we reach a 4-uniform $\mathcal{H}_4$ such that $S_{xy}(\mathcal{H}_4)\subseteq\mathcal{G}_2$ with core $\{x, x_1, x_2\}$, then there is a set $X_4=\{x, y, x_1, x_2, x_3\}$ such that after a series of shifts $S_{x'y'}$ $(x'<y'$ and $x', y'\in [n]\setminus X_4)$ to $\mathcal{H}_4$, we will reach a stable family $\mathcal{G}$ satisfying the conditions of Theorem \ref{thm1.5}. Moreover, $\{x, y, x_1, x_3\} \in \mathcal{G}$  and $G\cap X_4\ne \emptyset$ for any $G\in \mathcal{G}$.
\end{lemma}

\begin{proof}
Since $S_{xy}(\mathcal{H}_4)\subseteq\mathcal{G}_2$ with core $A$, by Fact \ref{fact 2} (ii), we have that $x\in A$ and $y\not \in A$. Let $A=\{x, x_1, x_2\}$. By the maximality of $|\mathcal{H}_4|$, we may assume
\begin{align*}
&\left\{ G\in {X \choose 4}: \{x_1, x_2\} \subseteq G \right\}\subseteq \mathcal{H}_4, \\
&\left\{G\in {X \choose 4}: \{x, y\} \subseteq G, G\cap \{x_1, x_2\} \ne \emptyset \right\}\subseteq \mathcal{H}_4.
\end{align*}
So
\begin{align}\label{eq13}
&\vert \left\{ G\in \mathcal{H}_4: \{x_1, x_2\} \subseteq G \right\}\vert={n-2 \choose 2}, \\ \label{eq14}
&\vert \left\{G\in \mathcal{H}_4: \{x, y\} \subseteq G, G\cap \{x_1, x_2\} \ne \emptyset \right\}\vert=2n-7.
\end{align}

Choose a  set $G=\{x, y, x_1, x_3\}\in \mathcal{H}_4$ and let $X_4:=\{x, y, x_1, x_2, x_3\}$. Since $S_{xy}(\mathcal{H}_4)\subseteq\mathcal{G}_2$ with core  $\{x, x_1, x_2\}$, every member in $\mathcal{H}_4$ intersects $X_4$.  Every member in $\mathcal{H}_4$ is stable at every element  in $X_4$ under   shifts $S_{x'y'}$ $(x'<y'$ and $x', y'\in [n]\setminus X_4)$. So $\{x, y, x_1, x_3\}$ is  in the final stable family $\mathcal{G}$  and $G\cap X_4\ne \emptyset$ for any $G\in \mathcal{G}$. What remains is to show that performing shifts $S_{x'y'}$ $(x'<y'$ and $x', y'\in [n]\setminus X_4)$ to $\mathcal{H}_4$ will not reach {\it Cases 1-5} in Remark \ref{remark2.4}.

 By (\ref{eq13}), for any $x'\in [n]\setminus X_4$, there are at least ${n-3 \choose 2}$ members in $\mathcal{H}_4$ missing $x'$,  so we can not reach {\it Cases 1-3}.

 Assume $\mathcal{H'}:=S_{x'y'}(\mathcal{H''})\subseteq\mathcal{G}_2$ with core $A'$. By (\ref{eq13}), Fact \ref{fact 2} (ii) and Claim \ref{clm 1}, $A'=\{x', x_1, x_2\}$. Since $G \in \mathcal{H'}$, and $|H\cap A'|=1$, we get a contradiction, hence we will not reach {\it Case 4}. At last, assume $\mathcal{H'}:=S_{x'y'}(\mathcal{H''})\subseteq\mathcal{G}_3$ with center and core forming a 4-set $B$. By Fact \ref{fact 2} (iv), $x'\in B$. By Claim \ref{clm 2} (ii) and (\ref{eq13}), $\{x_1, x_2\}\subseteq B$ and the center must be $x_1$ or $x_2$. That is $\{x_1, x_2, x'\}\subset B$. Since $|B|=4$, $|\{x, y\}\cap B|=0 $ or $1$. If $|\{x, y\}\cap B|=0 $, then the sets containing $\{x, y\}$ in $\mathcal{H'}$ must contain center and one point of core $A'$, so $d_{\{x, y\}}\le 3$. If $|\{x, y\}\cap B|=1$, then the sets containing $\{x, y\}$ in $\mathcal{H'}$ either contain center or contain core $A'$, so $d_{\{x, y\}}\le n-3+1=n-2$. These members containing $\{x, y\}$ in $\mathcal{H}_4$ are also in $\mathcal{H'}$, by (\ref{eq14}), there are at least $2n-7$, a contradiction. Hence we can not reach {\it Case 5}.
\end{proof}

\begin{lemma}[{\it Case 5}]\label{case5}
If we reach a 4-uniform $\mathcal{H}_5$ such that $S_{xy}(\mathcal{H}_5)\subseteq\mathcal{G}_3$ with center and core $E$ forming a 4-set $B$, then there is a set $X_5=\{x, y, x_1, x_2, x_3\}$ such that after a series of shifts $S_{x'y'}$ $(x'<y'$ and $x', y'\in [n]\setminus X_5)$ to $\mathcal{H}_5$, we will reach a stable family $\mathcal{G}$ satisfying the conditions of Theorem \ref{thm1.5}. Furthermore, $\vert G\cap X_5\vert\ge 2$ for each $G\in  \mathcal{G}$.
\end{lemma}

\begin{proof}
For $S_{xy}(\mathcal{H}_5)$, we will meet one of the following three cases. Case (a): $x$ is the center, $y \in E$, and $E=\{y, x_1, x_2\}$. In this case, we may assume that $ \{y, x_1, x_2, x_3\}\in S_{xy}(\mathcal{H}_5)$ for some $x_3\in [n]\setminus B$. Let $X_5:= \{x, y, x_1, x_2, x_3\}$.  Case (b): $x$ is the center, $y \not \in E$, and $E=\{ x_1, x_2, x_3\}$. In this case, let $X_5:= \{x, y, x_1, x_2, x_3\}$.  Case (c): $x_1$ is the center, $x \in E$, and $E=\{x, x_2, x_3\}, y\in [n]\setminus B$. In this case, let $X_5:= \{x, y, x_1, x_2, x_3\}$. We first observe that $\vert G\cap X_5\vert\ge 2$ for each $G\in  \mathcal{H}_5$ in each case.

First we consider Case (a). In this case, a member in $\mathcal{H}_5$ must contain $x$ or $y$.
By the maximality of $|\mathcal{H}_5|$, we may assume
\begin{align*}
\left\{ G\in {X \choose 4}: \{x, y\} \subseteq G \right\}\subseteq \mathcal{H}_5.
\end{align*}
So
\begin{align}\label{eq15}
\vert \left\{ G\in \mathcal{H}_5: \{x, y\} \subseteq G \right\}\vert={n-2 \choose 2}.
\end{align}

Performing $S_{x'y'}$ in $[n]\setminus X_5$ to $\mathcal{H}_5$ will not reach {\it Cases 1-3} since there are at least ${n-3 \choose 2}$ members that containing $\{x, y\}$ and missing $x'$ in $\mathcal{H}_5$ and these sets are stable after $S_{x'y'}$ in $[n]\setminus X_5$  (by (\ref{eq15})).

Assume that $\mathcal{H'}:=S_{x'y'}(\mathcal{H''})\subseteq\mathcal{G}_2$ with core $A$. By (\ref{eq15}), Fact \ref{fact 2} (ii) and Claim \ref{clm 1}, $A=\{x', x, y\}$. Since $S_{xy}(\mathcal{H}_5)$ is not EKR, $\{y, x_1, x_2, x_3\} \in S_{xy}(\mathcal{H}_5)$, $\{x, x_1, x_2, x_3\} \in \mathcal{H}_5$, so in $\mathcal{H'}$. However, $|\{x, x_1, x_2, x_3\}\cap A|=1$, this is a contradiction, hence we will not reach {\it Case 4}. Assume that $\mathcal{H'}:=S_{x'y'}(\mathcal{H''})\subseteq\mathcal{G}_3$ with center and core forming a 4-set $B'$. By (\ref{eq15}), Fact \ref{fact 2} (iv) and Claim \ref{clm 2} (ii),  $\{x, y, x'\}\subseteq B'$, and the center is either $x$ or $y$. In either case, the existence of $\{x, x_1, x_2, x_3\}$ and $ \{y, x_1, x_2, x_3\} $ will lead to a contradiction. Hence we will not reach {\it Case 5}.

Next we consider Case (b).
By the maximality of $|\mathcal{H}_5|$, we may assume that
$$\left\{G\in {X \choose 4}: \{x, y\} \subseteq G\,\,\text{and}\,\, G\cap E \ne \emptyset \right\}\subseteq \mathcal{H}_5$$
and
$$\left\{G\in {X \choose 4}: \{x_1, x_2, x_3\} \subseteq G \right\}\subseteq \mathcal{H}_5.$$
In particular, $\{x, x_1, x_2, x_3\}\in \mathcal{H}_5$ and $\{y, x_1, x_2, x_3\}\in \mathcal{H}_5$.
Computing directly, we have
\begin{equation}\label{eq16}
 \vert\left\{G\in \mathcal{H}_5: \{x, y\} \subseteq G, G\cap E \ne \emptyset \right\}\vert=3n-12
 \end{equation}
and
\begin{equation}\label{eq17}
 \vert\left\{G\in \mathcal{H}_5: \{x_1, x_2, x_3\} \subseteq G \right\}\vert=n-3.
\end{equation}
Since $\{x, y, x_1, x_2\}, \{x, y, x_1, x_3\}, \{x, y, x_2, x_3\}\in \mathcal{H}_5$ and these sets miss $x'$ and are stable after shifts $S_{x'y'}$ ($x' <y'$ and $x', y' \in [n]\setminus X_5$), we will not reach {\it Cases 1-3}.

Assume $\mathcal{H'}:=S_{x'y'}(\mathcal{H''})\subseteq\mathcal{G}_2$ with core $A$, where $x' <y'$ and $x', y' \in [n]\setminus X_5$.
By (\ref{eq16}), Fact \ref{fact 2} (ii) and Claim \ref{clm 1}, we have $A=\{x', x, y\}$. However $\{x, x_1, x_2, x_3\}\in \mathcal{H'}$ and $|\{x, x_1, x_2, x_3\}\cap A|=1$, a contradiction, so we will not reach {\it Case 4}.

 Assume that $\mathcal{H'}:=S_{x'y'}(\mathcal{H''})\subseteq\mathcal{G}_3$ with center and core forming a 4-set $B'$. By Fact \ref{fact 2} (iv), $x' \in B'$.
Equation (\ref{eq16}) and Claim \ref{clm 2} (i)  imply that the center must be $x$ or $y$.

By Claim \ref{clm00} and (\ref{eq17}), either $\{x_1, x_2, x_3\} \subset B'$ or $|\{x_1, x_2, x_3\} \cap B'|=2$ and one of $\{ x_1, x_2, x_3\}$ is the center. But it's impossible to satisfy both conditions in the previous paragraph and this paragraph, hence we will not reach {\it Case 5}.

At last we consider Case (c). By the maximality of $|\mathcal{H}_5|$, we may assume
\begin{align*}
\left\{G\in {X \choose 4}: \{x_1, x_2\} \subseteq G \right\}\subseteq \mathcal{H}_5 \,\,\,\,
\text{and}\,\,\,\,
\left\{G\in {X \choose 4}: \{x_1, x_3\} \subseteq G \right\}\subseteq \mathcal{H}_5.
\end{align*}
By direct computation,
\begin{align}\label{eq18}
&\vert \left\{ G\in \mathcal{H}_5: \{x_1, x_2\} \subseteq G \right\}\vert={n-2 \choose 2}, \\ \label{eq19}
&\vert \left\{G\in \mathcal{H}_5: \{x_1, x_3\} \subseteq G \right\}\vert={n-2 \choose 2}.
\end{align}
Since there are ${n-3 \choose 2}$ sets containing $\{x_1, x_2\}$ but missing $x'$, after performing $S_{x'y'}$ in $[n]\setminus X_5$ to $\mathcal{H}_5$, we will not reach {\it Case 1-3}.

If we reach {\it Cases 4}, that is, after performing shifts in $[n]\setminus X_5$ to $\mathcal{H}_5$ repeatedly, $S_{x'y'}(\mathcal{H'})\subseteq\mathcal{G}_2$ with core $A$. By (\ref{eq18}), (\ref{eq19}), Fact \ref{fact 2} (ii) and Claim \ref{clm 1}, $x', x_1, x_2, x_3\in A$, but $|A|=3$, a contradiction. If we reach {\it Case 5}, that is $S_{x'y'}(\mathcal{H'})\subseteq\mathcal{G}_3$ with the center and the core forming a 4-set $B'$. By (\ref{eq18}), (\ref{eq19}), Fact \ref{fact 2} (iv) and Claim \ref{clm 1},  $B'=\{x_1, x_2, x_3, x'\}$, and $x_1$ is the center. Recall that $\{x, y, x_2, x_3\}\in \mathcal{H}_5$, also in $S_{x'y'}(\mathcal{H'})$, but $\{x, y, x_2, x_3\}\cap \{x_1, x', y'\} = \emptyset$, a contradiction, hence we cannot reach {\it Case 5}.

As remarked earlier,  $\vert G\cap X_5\vert\ge 2$ for each $G\in  \mathcal{H}_5$. Note that performing shifts in $[n]\setminus X_5$ to $\mathcal{H}_5$ keeps this property, so $\vert G\cap X_5\vert\ge 2$ for each $G\in  \mathcal{G}$.
\end{proof}

By Lemmas \ref{case1} to \ref{case5}, we have shown that if one of {\it Case 1-5} happens, then there exists a set $X_i$ with $|X_i| \le 5$ and $\{x, y\} \subseteq X_i$ such that performing shifts in $[n]\setminus X_i$ to $\mathcal{H}_i$ will not result in any case of {\it Case 1-5}, so the final family is a stable family satisfying the conditions in Theorem \ref{thm1.5}. Furthermore, $G \cap X_i \ne \emptyset$ for any set $G$ in the final family. So we complete the proof of Lemma \ref{lem2.1}.
%\end{proof}

\subsection{Proof of Lemma \ref{lem2.2}}

\begin{proof}
 We first consider $k\ge5$. In this case, we have $|X_i|\le 3$ and  $|Y_i|\ge 2k-3$.

  We first prove (ii). Assume on the contrary that there are $G$ and $G'\in \mathcal{G}$ such that $G\cap G' \cap Y= \emptyset $ and let $|G\cap G'|$ be the minimum among all pairs of sets in $\mathcal{G}$ not intersecting in $Y$. Clearly $|G\cap G' \cap ([n]\setminus Y)|\ge1$. Note that $| (G\cup G' )\cap Y_i|\le |G\cap Y_i|+|G'\cap Y_i|\le 2k-4$ (since $|G\cap ([n]\setminus Y)|\ge1$ and $|G\cap X_i|\ge1$, so $|G\cap Y_i|\le k-2$, same for $G'$). But $|Y_i|\ge 2k-3$, so there exists a point $a\in Y_i\setminus (G\cup G' )$. Pick any point $b\in G\cap G' \cap ([n]\setminus Y)$, we have $a<b$. Notice that $\mathcal{G}$ is stable on $[n]\setminus X_i$, so $G'':=(G'\setminus \{b\})\cup \{a\}\in \mathcal{G}$. Then $G\cap G'' \cap Y= \emptyset$ and $|G\cap G''|<|G\cap G'|$, contradicting the minimality of $|G\cap G'|$.

 For (i), assume on the contrary, that $\mathcal{A}_1\ne \emptyset$. Let $\{x\}\in \mathcal{A}_1$, then there is a set $G\in \mathcal{G}$ such that $G\cap Y=\{x\}$. By (ii), for any another set $G'\in \mathcal{G}$ we have $G\cap G' \cap Y\ne \emptyset$, so $x\in G'$. This implies that $\mathcal{G}$ is EKR, a contradiction, so $\mathcal{A}_1=\emptyset$.

Next consider for $k=4$. In this case, for $1\le i\le 5$, $|X_i|=5$ and $|Y_i|=9-5=4$,  and for $i=6$, $|X_i|=0$ and $|Y_i|=9 $.
\begin{clm}\label{clm3.8}
If $G$ and $G'$ in $\mathcal{G}$ satisfies that $|Y_i\setminus (G\cup G')|\ge |G\cap G'\cap ([n]\setminus Y)|$, then $G\cap G' \cap Y\ne \emptyset$.
\end{clm}
\begin{proof}
If $G\cap G'\cap Y=\emptyset$, then $D:=G\cap G'\cap ([n]\setminus Y)\ne \emptyset$.
Since $|Y_i\setminus (G\cup G')|\ge$ $|G\cap G'\cap ([n]\setminus Y)|$, there is a subset $D' \subseteq Y_i\setminus (G\cup G')$ with size $|D'|=|D|$.  By the definition of $Y_i$, all numbers in $D'$ are smaller than $D$. Since $\mathcal{G}$ is stable on $[n]\setminus X_i$, $F:=(G'\setminus D)\cup D'\in \mathcal{G}$. However $G \cap F=\emptyset$, a contradiction to the intersecting property of $\mathcal{G}$. So $G\cap G'\cap Y \ne\emptyset$.
\end{proof}

\begin{clm}\label{clm3.9}
$ |\mathcal{A}_1|\le 1$; $\mathcal{A}_2$ and $\mathcal{A}_4$ are intersecting.
\end{clm}

\begin{proof}
Obviously, $\mathcal{A}_4$ is intersecting. Assume that $|\mathcal{A}_1|\ge 2$ and $\{x_1\}, \{x_2\}\in \mathcal{A}_1$.  Then there are $G$ and $G'$ in $\widetilde{\mathcal{A}_1}$ such that $G\cap Y=\{x_1\}$ and $G'\cap Y=\{x_2\}$. Since any set in $\mathcal{G}$ intersects with $X_i$ (for $i\in [5]$), $x_1, x_2 \in X_i$. So $1\le |G\cap G'\cap ([n]\setminus Y)|\le 3< 4=|Y_i\setminus (G\cap G')|$. By Claim \ref{clm3.8}, $G\cap G'\cap Y\ne \emptyset$, a contradiction. Hence, $ |\mathcal{A}_1|\le 1$. Let $G$ and $G'$  be in $\widetilde{\mathcal{A}_2}$. Then $|G\cap G'\cap ([n]\setminus Y)|\le 2$. Since  $|G\cap X_i|\ge 1$ and $|G'\cap X_i|\ge 1$ (for $i\in [5]$), then $|Y_i\setminus (G\cup G')|\ge 2$. By Claim \ref{clm3.8}, $G\cap G'\cap Y\ne \emptyset$, that is $\mathcal{A}_2$ is intersecting, as desired.
\end{proof}

\begin{clm}
$\mathcal{A}_1=\emptyset$.
\end{clm}

\begin{proof}
 By Claim \ref{clm3.9}, $\vert\mathcal{A}_1\vert \le 1$. We may assume on the contrary that $\mathcal{A}_1=\{\{x\}\} $ for some $x\in X_i$.  For any $G \in \widetilde{\mathcal{A}_1}$ and $G'\in \widetilde{\mathcal{A}_j}$ (for $j=2, 3, 4$), $G$ and $G'$ satisfy the condiction of Claim \ref{clm3.8}, so $G\cap G'\cap Y\ne \emptyset$, this implies that $x\in G'$ and hence $\mathcal{G}$ is EKR, a contradiction.
\end{proof}

\begin{clm}\label{clm3.10}
$\mathcal{A}_2$ and $\mathcal{A}_3$ are cross-intersecting.
\end{clm}

\begin{proof}
Let $G\in \widetilde{\mathcal{A}_2}$ and $G'\in \widetilde{\mathcal{A}_3}$.
Then $|G \cap G' \cap ([n] \setminus Y)|\le 1$.
Since any set in $\mathcal{G}$ intersects with $X_i$ (for $i \in [5]$), $|Y_i\setminus (G\cup G')|\ge 1$. By Claim \ref{clm3.8}, $G\cap G'\cap Y\ne \emptyset$, that is $\mathcal{A}_2$ and $\mathcal{A}_3$ are cross-intersecting, as desired.
\end{proof}

\begin{clm}\label{clm3.11}
$\mathcal{A}_3$ is intersecting.
\end{clm}

\begin{proof}
 Assume on the contrary, that there exist $A$, $A'\in \mathcal{A}_3$ and $G$, $G' \in \widetilde{\mathcal{A}_3}$ such that $G\cap Y=A$, $G'\cap Y=A'$ and $A\cap A'=\emptyset$, in other words, $G\cap G'\cap Y=\emptyset$ and $|G\cap G'\cap ( [n]\setminus Y)|=1$. If $|(G\cup G')\cap Y_i|\le 3$, by Claim \ref{clm3.8}, $G\cap G'\cap Y\ne\emptyset$, a contradiction. Hence we only need to consider the following case : $|A\cap X_i|=1, |A\cap Y_i|=2, |A'\cap X_i|=1$ and $|A'\cap Y_i|=2$. Now we show the conclusion for each case of Lemma \ref{lem2.1}. All sets below are inherited from the proof of Lemma \ref{lem2.1} for each cooresponding case.

%For {\it Cases 1-2} in Lemma \ref{lem2.1}, we may assume $A\cap X_i=\{x\}$ and $A'\cap X_i=\{y\}$. Let $G\in \mathcal{G}$ such that $G\cap Y=A$ and $G'\in \mathcal{G}$ such that $G'\cap Y=A'$.

If we meet {\it Cases 1} in Lemma \ref{lem2.1}, then by Lemma \ref{case1}, we have that $X_1=\{x, y, y', z_1, w_1\}$ or $X_1=\{x, y, y', z_2, w_2\}$, and we may assume that $G\cap X_i=\{x\}$ and $G'\cap X_i=\{y\}$. Respectively,  $\{y, y', z_1, w_1\}$ or $\{x, y', z_2, w_2\}$  $\in \mathcal{G}$, which is disjoint from $G$ or $G'$. A contradiction to the intersecting property of $\mathcal{G}$.

If we meet {\it Cases 2}\,  in Lemma \ref{lem2.1}, then by Lemma \ref{case2}, we have that $X_2=\{x, y, z_1, z_2, z_3\}$, and either $\{z_1, z_2, z_3, y\}\in \mathcal{G}$ or $\{z_1, z_2, z_3, z_4'\}\in \mathcal{G}$  for some $y \neq z_4'$. Furthermore, if $\{z_1, z_2, z_3, y\}\in \mathcal{G}$, then every member in $\mathcal{G}$ contains $x$ or $y$. So we may assume that $G\cap X_i=\{x\}$ and $G'\cap X_i=\{y\}$.  Then
$\{z_1, z_2, z_3, y\}\cap G=\emptyset$, a contradiction. If $\{z_1, z_2, z_3, z'_4\}\in \mathcal{G}$, then every other member in $\mathcal{G}$ contains $x$ or $y$,  we may assume that $G\cap X_i=\{x\}$ and $G'\cap X_i=\{y\}$. Since $\mathcal{G}$ is stable, we may assume that
 $z'_4 \in Y_i$. Recall that $|G\cap Y_i|=|G'\cap Y_i|=2$, hence $\{z_1, z_2, z_3, z'_4\}$ must be disjoint from $G$ or $G'$,  a contradiction.

If we meet {\it Cases 3}\,  in Lemma \ref{lem2.1}, then by Lemma \ref{case3},  $|G\cap G'|\ge 2$, a contradiction.

If we meet {\it Cases 4}\,  in Lemma \ref{lem2.1}, then by Lemma \ref{case4}, we have that $X_4=\{x, y, x_1, x_2, x_3\}$, $\{x, y, x_1, x_3\}\in \mathcal{G}$ and $S_{xy}(\mathcal{H}_4)\subseteq\mathcal{G}_2$ with core $\{x, x_1, x_2\}$. So for every set $F$ in $\mathcal{H}_4$, either $|F\cap \{x, x_1, x_2\}|\ge 2$, or $F\cap \{x, x_1, x_2\}=\{x_1\}$ and $y\in F$, or $F\cap \{x, x_1, x_2\}=\{x_2\}$ and $y\in F$. In all cases, $\vert F\cap X_4\vert\ge 2$. Performing shifts in $[n]\setminus X_4$ will not change these properties, hence every set in $\mathcal{G}$ also has the same properties, in particular, $G$ and $G'$ do. This makes a contradiction to $\vert G\cap X_4\vert=\vert G'\cap X_4\vert=1$.

Assume that we meet {\it Case 5} in Lemma \ref{lem2.1},  then by Lemma \ref{case5}, we have that  $\vert G\cap X_5\vert\ge 2$ for each $G\in  \mathcal{G}$. This makes a contradiction to $\vert G\cap X_5\vert=\vert G'\cap X_5\vert =1$.

 At last, assume that we will not  meet any of  Cases 1-5 in Lemma \ref{lem2.1} if we perform shifts repeatedly to $\mathcal{G}$. In this case,  $Y=[2k]$.
 Assume on the contrary, and let $G$ and $G'\in \mathcal{G}$ such that $G\cap G' \cap Y= \emptyset $ and $|G\cap G'|$ is the minimum among all pairs of sets in $\mathcal{G}$ not intersecting in $Y$. Then $|G\cap G' \cap (X\setminus Y)|\ge1$. Consequently, $| (G\cup G' )\cap Y|\le |G\cap Y|+|G'\cap Y|\le 2k-2$ since $|G\cap Y|\le k-1$ and $|G'\cap Y|\le k-1$. So there exists a point $a\in Y\setminus (G\cup G' )$. Pick any point $b\in G\cap G' \cap (X\setminus Y)$. Note that $a<b$, then $G'':=(G'\setminus \{b\})\cup \{a\}\in \mathcal{G}$ since $\mathcal{G}$ is stable. It is easy to see that $G\cap G'' \cap Y= \emptyset$ and $|G\cap G''|<|G\cap G'|$, contradicting the minimality of $|G\cap G'|$.
\end{proof}

 Since $\mathcal{G}$ is intersecting, $\mathcal{A}_2$ and $\mathcal{A}_4$ are cross-intersecting, and $\mathcal{A}_3$ and $\mathcal{A}_4$ are cross-intersecting. Combining with Claims \ref{clm3.9}, \ref{clm3.10} and \ref{clm3.11}, we have completed the proof of (ii).
\end{proof}

\section{Concluding remarks}
It is natural to ask  what is the maximum size of a $k$-uniform intersecting family $\mathcal{F}$ with $\tau(\mathcal{F})\ge 3$. About this problem,  Frankl \cite{Fra1980} gave an upper bound for sufficient large $n$. To introduce the result, we need the following construction.

\begin{construction}
Let $x\in [n]$, $Y\subseteq [n]$ with $|Y|=k$, and $Z\subseteq [n]$ with $|Z|=k-1$, $x\not\in Y\cup Z$, $Z\cap Y=\emptyset$ and $Y_0=\{y_1, y_2\}\subseteq Y$. Define
\begin{gather*}
\mathcal{G}=\{A\subseteq [n]: x\in A, A\cap Y\ne \emptyset \,\,\text{and}\,\, A\cap Z\ne \emptyset\}\cup\{Y, Z\cup \{y_1\}, Z\cup \{y_2\}, \{x, y_1, y_2 \}\},\\
FP(n, k)=\{F\subseteq [n]: |F|=k, \exists G\in \mathcal{G} \,s.t., G\subseteq F\}.
\end{gather*}
\end{construction}
It is easy to see that $FP(n, k)$ is intersecting and $\tau(FP(n, k))=3$.

\begin{theorem}[Frankl \cite{Fra1980}]
Let $k\ge 3$ and $n$ be sufficiently large integers. Let  $\mathcal{H}$ be an $n$-vertex $k$-uniform family with $\tau(\mathcal{H})\ge 3$. Then $|\mathcal{H}|\le |FP(n, k)|$. Moreover, for $k\ge 4$, the equality holds only for $\mathcal{H}=FP(n, k)$.˙
\end{theorem}

It is interesting  to consider what is the maximum  $k$-uniform intersecting families with covering number  $s\ge 4$.

\section{Acknowledgements}
This research is supported by  National natural science foundation of China (Grant No. 11931002).

\end{document}